\documentclass[12pt,oneside,a4paper]{amsart}

\usepackage[backend=biber,bibstyle=sternstyle,dashed=true]{biblatex} 
\usepackage[utf8]{inputenc}
\usepackage[english]{babel}
\usepackage{amsmath}
\usepackage{amsfonts}
\usepackage{amsthm}
\usepackage{amssymb}
\usepackage{graphicx}
\usepackage{xcolor}
\usepackage{enumitem}
\usepackage{relsize}
\usepackage{hyperref}
\usepackage{tikz}
\usetikzlibrary{calc,decorations.markings}
\usepackage{tikz-cd}
\usepackage{todonotes}
\usepackage{geometry}
\makeatletter
\def\l@subsection{\@tocline{2}{0pt}{2.5pc}{5pc}{}}
\makeatother

\newcommand{\stoptocwriting}{%
	\addtocontents{toc}{\protect\setcounter{tocdepth}{-5}}}
\newcommand{\resumetocwriting}{%
	\addtocontents{toc}{\protect\setcounter{tocdepth}{\arabic{tocdepth}}}}

\makeatletter
\newtheorem*{rep@theorem}{\rep@title}
\newcommand{\newreptheorem}[2]{%
\newenvironment{rep#1}[1]{%
 \def\rep@title{#2 \ref{##1}}%
 \begin{rep@theorem}}%
 {\end{rep@theorem}}}
\makeatother

\newreptheorem{theorem}{Theorem}
\newcommand{\pgfextractangle}[3]{%
	\pgfmathanglebetweenpoints{\pgfpointanchor{#2}{center}}
	{\pgfpointanchor{#3}{center}}
	\global\let#1\pgfmathresult  
}
\usepackage{caption} %figures with hyperref
\definecolor{winered}{rgb}{0.5,0,0}
\hypersetup{
	colorlinks = true,
	urlcolor = winered,
	citecolor = winered,
	linkcolor = winered,
	pdfauthor = {Walker Stern},
	pdfkeywords = {},
	pdftitle = {2-Segal Objects in Spans},
	pdfsubject = {Article},
	pdfpagemode = UseNone
}
\usetikzlibrary{matrix, arrows}
\makeatletter
\tikzset{
	edge node/.code={%
		\expandafter\def\expandafter\tikz@tonodes\expandafter{\tikz@tonodes #1}}}
\makeatother
\tikzset{
	subseteq/.style={
		draw=none,
		edge node={node [sloped, allow upside down, auto=false]{$\subseteq$}}},
	Subseteq/.style={
		draw=none,
		every to/.append style={
			edge node={node [sloped, allow upside down, auto=false]{$\subseteq$}}}
	}
}
\usepackage{euscript}

\def\on{\operatorname}
\def\CY{{\on{CY}}}

\def\op{{\on{op}}}
\def\CC{\EuScript{C}}
\def\DD{\EuScript{D}}

\def\Fin{\on{Fin}}
\def\Ass{\mathcal{A}\!\on{ss}}
\def\bFin{\mathbb{F}\!\!\on{in}}
\def\bfLambda{\boldsymbol{\Lambda}}
\def\Cat{\on{Cat}}
\def\Hom{\on{Hom}}
\def\Set{\on{Set}}
\def\Tw{\on{Tw}}
\def\Span{\on{Span}}
\def\id{{\on{id}}}
\def\Fun{\on{Fun}}

\def\SAlg{\on{Alg}_{\on{Sp}}}
\def\res{\on{res}}
\def\tSeg{2\on{-Seg}}

\usepackage[capitalize,nameinlink]{cleveref}
\crefformat{equation}{(#2#1#3)}
\usepackage{xypic}
\usepackage{pifont}

\theoremstyle{theorem}
\newtheorem{prop}{Proposition}[section]
\newtheorem{cor}[prop]{Corollary}
\newtheorem{lem}[prop]{Lemma}
\newtheorem{thm}[prop]{Theorem}
\newtheorem*{thm*}{Theorem}
\newtheorem{theorem}{Theorem}

\theoremstyle{definition}
\newtheorem{deft}[prop]{Definition}
\newtheorem{rem}[prop]{Remark}
\newtheorem{cnstr}[prop]{Construction}

\newenvironment{defn}
{%
	\pushQED{\qed}\begin{deft}}
	{\popQED\end{deft}}

\newenvironment{rmk}
{%
	\pushQED{\qed}\begin{rem}}
	{\popQED\end{rem}}
	
\newenvironment{const}
{%
	\pushQED{\qed}\begin{cnstr}}
	{\popQED\end{cnstr}}

\usepackage[bbgreekl]{mathbbol}
\DeclareMathSymbol\bbDelta  \mathord{bbold}{"01}
\DeclareMathSymbol\bbLambda \mathord{bbold}{"03}
\DeclareMathSymbol\bbGamma \mathord{bbold}{"00}
\DeclareMathSymbol\bbOne \mathord{bbold}{'061}

\DeclareRobustCommand{\bbNabla}{\text{\raisebox{\depth}{\scalebox{1}[-1]{$\bbDelta$}}}}
\title{2-Segal objects and algebras in spans}
\author{Walker H. Stern}

\begin{document}
\maketitle

\begin{abstract}
We define a category parameterizing Calabi-Yau algebra objects in an infinity category of spans. Using this category, we prove that there are equivalences of infinity categories relating, firstly: 2-Segal simplicial objects in C to algebra objects in Span(C); and secondly:  2-Segal cyclic objects in C to Calabi-Yau algebra objects in Span(C). 
\end{abstract}
\tableofcontents

\section*{Introduction}

\stoptocwriting
\subsection*{2-Segal objects and associativity}
A familiar concept in higher category theory is that of \emph{Segal objects} in an $\infty$-category $\CC$, that is, simplicial objects $X:\Delta^\op\to \CC$ such that the natural map 
\[
X_n\to \overbrace{X_1\times_{X_0}X_1\times_{X_0}\cdots \times_{X_0} X_1}^{\times n}.
\]
is an equivalence. Introduced by Rezk in \cite{Rezk}, Segal objects show up in a variety of guises, from monoidal $\infty$-categories (cf. \cite{LurieDAGII}) to the nerves of 1-categories. Of particular interest is the algebraic content of the Segal condition. Given a Segal set $X$, the span 
\begin{equation}\label{eq:SegalCondition}
X_1\times_{X_0}X_1\leftarrow X_2\to X_1
\end{equation}
can be read as a multiplication law, owing to the invertibility of the left hand morphism. Moreover, the Segal condition on higher simplices also expresses the associativity of this multiplication. 

%The Segal condition may be viewed in a geometric light, namely, that by labeling 

The Segal condition on a simplicial set was generalized to the `higher-dimensional' \emph{2-Segal condition} by Dyckerhoff and Kapranov \cite{DKHSSI} and Gálvez-Carrillo, Kock, and Tonks \cite{GKT} (2-Segal spaces are called \emph{decomposition spaces} in the latter). In a sense, the 2-Segal condition no longer requires that the span \eqref{eq:SegalCondition} define a multiplication, but retains the higher associativity conditions encoded in the higher simplices. More precisely, the 2-Segal condition on a simplicial object $X:\Delta^\op\to \CC$ requires that the diagrams 
\[
\begin{tikzcd}
X_n\arrow[r]\arrow[d] & X_{i,\ldots,j\arrow[d]}\\
X_{1,\ldots i,j,\ldots}\arrow[r] & X_{i,j}
\end{tikzcd}
\] 
all be pullback diagrams in $\CC$. The 2-Segal condition is indeed a generalization of the Segal condition, insofar as every Segal simplicial object is 2-Segal. 

A 2-Segal object $X$ is said to be \emph{unital} if, additionally, the diagrams 
\[
\begin{tikzcd}
X_{n-1}\arrow[r]\arrow[d] & \arrow[d]X_n\\
X_{i}\arrow[r] & X_{i,i+1}
\end{tikzcd}
\] 
are pullback in $\CC$. Throughout this paper, we will use the term `2-Segal object' to refer to a \emph{unital} 2-Segal object in the terminology of \cite{DKHSSI}. 

The sense in which such structures encode associativity relies on thinking of spans 
\[
\begin{tikzcd}
X_1\times X_1\times \cdots \times X_1 & \arrow[l] X_n\arrow[r] & X_1 
\end{tikzcd}
\]
as `$n$-fold multiplications', regardless of whether the left-hand morphisms are equivalences. We then compose by concatenating spans and taking a pullback, thinking of the result as a `space of compositions'. In this language, the 2-Segal condition says that the space of compositions of $n$-fold multiplications with $m$-fold multiplications is precisely the space of $n+m-1$-fold multiplications.  

There are a number of ways to make this intuitive picture rigorous (see, for example, the relation to Hall algebras presented in \cite{DKHSSI}, and the connection with operads from \cite{Walde}). The present paper concerns itself with one such perspective, namely, considering the relation between 2-Segal objects in an $\infty$-category $\CC$ and algebra objects in an $\infty$-category $\Span(\CC)$ whose morphisms are spans in $\CC$. Several results in this direction have already appeared in the literature. In the original Dyckerhoff-Kapranov paper \cite{DKHSSI}, monads and algebra objects in $(\infty,2)$-categories of spans were constructed from 2-Segal objects. More recently, Penney \cite{Penney} defined lax algebras in spans coming from simplicial objects, and showed that the associativity of these lax algebras was equivalent to the 2-Segal condition. In this paper, we restrict to $\infty$-categories of spans, and prove:
\begin{theorem}
Let $\CC$ be an $\infty$-category with small limits. There is an equivalence of $\infty$-categories 
\[
\left\lbrace\substack{\text{Algebra objects}\\ \text{in }\Span(\CC)} \right\rbrace\simeq \left\lbrace \substack{\text{2-Segal simplical} \\ \text{objects in } \CC}\right\rbrace.
\]
\end{theorem}

This theorem appears in full detail in the text as \cref{thm:AlgInSpan}. The functoriality of \cref{thm:AlgInSpan} is somewhat unusual. The $\infty$-category $\Span(\CC)$ is defined via an adjunction as is \cite{DKHSSI}. Morphisms of algebra objects in $\Span(\CC)$ are then defined to be natural transformations of the corresponding adjoint diagram in $\CC$, rather than natural transformations in $\Span(\CC)$.

\subsection*{Polygons, surfaces, and topological field theories}

There is an additional geometric intuition underlying the 2-Segal condition. Fix a standard $n+1$-gon $P_n$, and a simplicial object $X:\Delta^\op\to \CC$.  The set of vertices $V_n$ of $P_n$ defines a simplicial set $\Delta^{V_n}$. For any triangulation $\mathcal{T}$ of $P_n$ with vertices in $P_n$, one can define a simplicial subset $\Delta^{\mathcal{T}}\subset \Delta^{V_n}$ whose 2-simplices  correspond to the triangles in $\mathcal{T}$. Taking limits of the simplicial object $X$ over the corresponding categories of simplices,  the inclusion $\Delta^{\mathcal{T}}\subset \Delta^{V_n}$  yields a morphism
\[
X_n\to \lim_{\Delta_{/\Delta^{\mathcal{T}}}} X_k.
\]
By \cite[Proposition 2.3.2]{DKHSSI}, the 2-Segal condition is equivalent to the condition that this morphism be an equivalence for every $n\geq 2$ and every such triangulation $\mathcal{T}$ of $P_n$. Intuitively, this means that the 2-Segal condition allows one to glue together the $X_n$ to get invariants of 2-dimensional simplicial complexes.

The connection of 2-Segal spaces to 2-dimensional geometry can be extended further with recourse to 2-Segal \emph{cyclic} objects, that is cyclic objects in $\CC$ whose underlying simplicial objects are 2-Segal. In \cite[Section V.2]{DKCSGSS}, Dyckerhoff and Kapranov construct invariants $X(S,M)$ of stable marked surface $(S,M)$ with boundary, associated to a  2-Segal cyclic object $X:\Lambda^\op\to \CC$. For the subset $N\subset M$ of marked points on the boundary of $S$, this invariant comes equipped with a projection $X(S,M)\to X_1^{|M|}$. More suggestively, if we label some of these marked points as `incoming' and the rest as `outgoing', we can read the invariant $X(S,M)$ as a span 
\[
X_1^{|N_{\on{in}}|}\leftarrow X(S,M)\rightarrow X_1^{|N_{\on{out}}|}.
\]
Moreover, the $X(S,M)$ come equipped with coherent actions of the mapping class group. It is therefore natural to ask whether the invariants $X(S,M)$ form an open, oriented, $\infty$-categorical topological field theory in $\Span(\CC)$. 

Such open, oriented theories have attracted some attention in the literature already. In \cite{Costello}, Costello considers open oriented theories equipped with a set of D-branes and valued in the (dg-)category of chain complexes. He shows that such field theories are equivalent to Calabi-Yau $A_\infty$ categories --- a generalization of the Calabi-Yau algebras in chain complexes. A similar classification which has more bearing on the situation detailed above, is that of Lurie: 
\begin{thm*}[{\cite[Theorem 4.2.11]{LurieTFT}}]
Let $\CC$ be a symmetric monoidal $\infty$-category. The following types of data are equivalent:
\begin{enumerate}
\item Open oriented topological field theories in $\CC$.
\item Calabi-Yau algebra objects in $\CC$.
\end{enumerate}
\end{thm*}

Based on this theorem, the latter half of this paper seeks to relate cyclic 2-Segal objects to Calabi-Yau algebras. Such a relation is realized by:
\begin{theorem}
Let $\CC$ be an $\infty$-category with small limits. There is an equivalence of $\infty$-categories 
\[
\left\lbrace\substack{\text{Calabi-Yau}\\\text{Algebra objects}\\ \text{in }\Span(\CC)} \right\rbrace\simeq \left\lbrace \substack{\text{2-Segal cyclic} \\ \text{objects in } \CC}\right\rbrace.
\]
\end{theorem}

This appears in the text in full detail as \cref{thm:CYAlgsSpan}. As a consequence of \cref{thm:CYAlgsSpan}, we see that 2-Segal cyclic objects in $\CC$ are equivalent to open oriented topological field theories in $\CC$.

\subsection*{Examples and consequences}

Once the correspondence of \cref{thm:CYAlgsSpan} is established, a wealth of avenues to construct topological field theories open up. A number of examples of interest have already been explored in the literature.

\begin{itemize}
	\item Per \cite{DKTSTC}, the Waldhausen S-construction also gives rise to many cyclic 2-segal spaces.  An interesting special case is discussed in \cite{DKTSTC,DyckerhoffTopFuk, DKCSGSS}, where various versions of topological Fukaya categories are constructed as invariants $X(S,M)$ associated to 2-Segal objects arising from the Waldhausen S-construction.
	\item 1-Segal cyclic objects also provide a zoo of interesting examples. As a particular example, consider a morphism $f:A\to B$ in the $\infty$-category of spaces $\EuScript{S}$. The \v{C}ech nerve of this morphism is the 1-Segal simplicial space 
	\[
	\begin{tikzcd}
	\cdots & A\times_B A\times_B A\arrow[r,shift left=0.75ex]\arrow[r,shift right=0.75ex]\arrow[r] & A\times_B A\arrow[r,shift left=0.75ex]\arrow[r,shift right=0.75ex] & A 
	\end{tikzcd}
	\]
	which realizes to $B$. An appropriately chosen circle action on $B$ equips the \v{C}ech nerve with a canonical cyclic structure, and similarly, a cyclic structure on the \v{C}ech nerve equips its realization with a coherent $S^1$-action. Loosely speaking, the surface invariant $X(S,M)$ associated to this cyclic \v{C}ech nerve of $f$ is the space of `$S^1$-equivariant $B$-local systems on the circle bundle of a twisted tangent bundle of $(S,M)$ equipped with reduction of structure group to $A$ over the marked points'. When $B$ is $B SL_2(\mathbb{R})$ and $A$ is $BU$, where $U$ is the subgroup of upper unitriangular matrices, this construction can be related to the higher Teichm\"uller spaces constructed by Fock and Goncharov in \cite{FG}. 
	\item Another interesting incarnation of the cyclic \v{C}ech nerve construction is its application to a morphism $f:\ast \to X$
	into a connected space $X$. In this context, the \v{C}ech nerve has the loop space $\Omega X$ based at $f(\ast)$ as its space of $1$-simplices, and we expect the resulting surface invariants to relate to string topology.
\end{itemize}

%1-Segal cyclic objects also provide a zoo of interesting examples. As a particular example, consider a morphism $f:A\to B$ in the $\infty$-category of spaces $\EuScript{S}$. The \v{C}ech nerve of this morphism is the 1-Segal simplicial space 
%\[
%\begin{tikzcd}
%\cdots & A\times_B A\times_B A\arrow[r,shift left=0.75ex]\arrow[r,shift right=0.75ex]\arrow[r] & A\times_B A\arrow[r,shift left=0.75ex]\arrow[r,shift right=0.75ex] & A 
%\end{tikzcd}
%\]
%which realizes to $B$. An appropriately chosen circle action on $B$ equips the \v{C}ech nerve with a canonical cyclic structure, and similarly, a cyclic structure on the \v{C}ech nerve equips its realization with a coherent $S^1$-action. Loosely speaking, the surface invariant $X(S,M)$ associated to this cyclic \v{C}ech nerve of $f$ is the space of `$S^1$-equivariant $B$-local systems on the circle bundle of a twisted tangent bundle of $(S,M)$ equipped with reduction of structure group to $A$ over the marked points'. When $B$ is $B SL_2(\mathbb{R})$ and $A$ is $BU$, where $U$ is the subgroup of upper unitriangular matrices, this construction can be related to the higher Teichm\"uller spaces constructed by Fock and Goncharov in \cite{FG}. 

\cref{thm:CYAlgsSpan} and \cref{thm:AlgInSpan} also bear an interesting relation to another construction in the literature. Following Cisinski and Moerdijk (cf. \cite{CM}), Walde defines a notion of a cyclic $\infty$-operad in \cite{Walde}, and shows that there are equivalences of $\infty$-categories
\[
\left\lbrace \substack{\text{invertible cyclic}\\
\infty\text{-operads}}\right\rbrace\simeq \left\lbrace \substack{\text{2-Segal cyclic} \\ \text{objects in } \EuScript{S}}\right\rbrace
\]
and
\[
\left\lbrace \substack{\text{invertible}\\
\infty\text{-operads}}\right\rbrace\simeq \left\lbrace \substack{\text{2-Segal simplicial} \\ \text{objects in } \EuScript{S}}\right\rbrace.
\]
Which now has the immediate implication of relating invertible (cyclic) $\infty$-operads to (Calabi-Yau) algebras in $\Span(\EuScript{S})$. 

There are also a number of possible generalizations of Theorems \ref{thm:AlgInSpan} and \ref{thm:CYAlgsSpan}. For instance, the cyclic category $\Lambda$ is one example of a \emph{crossed simplicial group}, a notion defined by Fiedorowicz and Loday \cite{FL} and Krasauskas \cite{Kras}. In \cite{DKCSGSS}, invariants analogous to the $X(S,M)$ were constructed for functors $X:\Delta\mathfrak{G}^\op\to \CC$ satisfying the 2-Segal condition, where $\Delta\mathfrak{G}$ is a crossed simplicial group. We expect that the relation between open topological field theories in spans and 2-Segal cyclic objects generalizes to this additional structure, which will be the basis for some future work on the subject.

\subsection*{Acknowledgements}

I thank my doctoral advisor, Tobias Dyckerhoff for his advice and guidance. I am also grateful to the Max Planck Institute for Mathematics in Bonn and the Universit\"at Hamburg for supporting my studies.

\resumetocwriting

\section{The menagerie: notations, conventions, and constructions}

In this section, we will lay out the fundamental definitions and constructions that will be used in the proof of the main result. Along the way, we will also prove basic relations between these definitions, to alleviate the density of later arguments.
\subsection{Linear and cyclic orders}
\begin{defn}\label{defn:SimplexCats}
The \emph{simplex category} $\Delta$ has objects the standard linearly ordered sets $[n]=\{0,1,\ldots, n\}$ for $n\geq 0$ and morphisms the order-preserving maps. The \emph{enlarged simplex category} $\bbDelta$ has objects finite non-empty linearly ordered sets, and morphisms order-preserving maps. 

The \emph{augmented simplex category} $\Delta_+$ (resp. the \emph{augmented simplex category} $\bbDelta_+$) is obtained from  $\Delta$ (resp. $\bbDelta$) by appending an initial object $\emptyset$, which will also sometimes be denoted by $[-1]$.

The \emph{interval category} $\nabla$ is the subcategory of $\Delta$ on the objects $[n]$ for $n\geq 1$, the morphisms of which preserve maximal and minimal elements. The \emph{enlarged interval category} $\bbNabla$ is the subcategory of $\bbDelta$ on those sets of cardinality $\geq 2$, whose morphisms preserve maximal and minimal elements.

The \emph{augmented interval category} $\nabla_+$ (resp. the \emph{augmented extended interval category} $\bbNabla_+$) is the subcategory of $\Delta$ (resp. $\bbDelta$) whose objects have cardinality $\geq 1$ and whose morphisms preserve the maximal and minimal elements. 
\end{defn}

\begin{defn}\label{defn:SetCats}
	The category of the standard finite sets $\underline{n}:=\{1,2,\ldots, n\}$ for $n\geq 0$ will be denoted $\Fin$. The category of the standard finite pointed sets $\langle n\rangle:=\underline{n}\amalg \{\ast\}$ will be denoted $\Fin_\ast$. The category of all finite sets (resp. the category of all finite points sets) will be denoted by $\bFin$ (resp. by $\bFin_\ast$). When convenient, we will denote by $\bbGamma$ (resp. by $\Gamma$) the opposites of the categories $\bFin_\ast$ (resp. $\Fin_\ast$). Given a pointed set $S\in \bFin_\ast$, we denote by $S^\circ$ the set $S\setminus \{\ast\}$, where $\ast$ denotes the basepoint of $S$. 
	
	We additionally denote by $\Ass$ the \emph{associative operad}, i.e. the category whose objects are objects of $\bFin_\ast$, and whose morphisms $\phi:S\to T$ are morphisms in $\bFin_\ast$ equipped with a chosen linear order on the fiber $\phi^{-1}(i)$ for each $i\in T^\circ$. Composition is defined by composition in $\bFin_\ast$, together with the lexicographic orders. Note that there is a forgetful functor $\Ass\to \bFin_\ast$, which equips $N(\Ass)$ with the structure of an $\infty$-operad in the sense of \cite{LurieHA}. 
\end{defn}

\begin{const}[Linear interstices]\label{const:LinInterstice}
	Given a linearly ordered set $S\in \bbDelta$ we define an \emph{inner interstice} of $S$ to be an ordered pair $(k,k+1)\in S\times S$, where $k+1$ denotes the successor to $k$. The set of inner interstices of $S$ is, itself, a linearly ordered set, with the order
	\[
	(k,k+1)\leq (j,j+1)\Leftrightarrow k\leq j
	\]
	We will denote the linearly ordered set of inner interstices of $S$ by $\mathbb{I}(S)$. Note that $\mathbb{I}([0])=\emptyset$. 
	
	Given a linearly ordered set $S\in \bbDelta_+$, let $\hat{S}$ be the set $\{a\}\amalg S\amalg \{b\}$, where $b$ is taken to be maximal and $a$ minimal. We define an \emph{outer interstice} of $S$ to be an inner interstice of $\hat{S}$. We will denote the linearly ordered set of outer interstices of $S$ by $\mathbb{O}(S)$. Note that $\mathbb{O}(\emptyset)=\{(a,b)\}$. 
	
	We define functors 
	\[
	\mathbb{O}:\bbDelta_+^{\op}\to \bbNabla_+;\quad S\mapsto \mathbb{O}(S)
	\]
	and 
	\[
	\mathbb{I}: \bbNabla_+^\op\to \bbDelta;\quad S\mapsto \mathbb{I}(S)
	\]
	as follows (we will define $\mathbb{O}$ explicitly, the definition of $\mathbb{I}$ is similar). Given a morphism $f:S\to T$ in $\bbDelta_+$, we define a morphism $\mathbb{O}(f):\mathbb{O}(S)\to \mathbb{O}(T)$ by setting 
	\[
	\mathbb{O}(f)(j,j+1)=\begin{cases}
	(k,k+1) & f(k)\leq j\leq j+1\leq f(k+1)\\
	(a,a+1) & j\leq f(k)\; \forall k\in S\\
	(b-1,b) & j\geq f(k)\; \forall k\in S. 
	\end{cases}
	\]
	Pictorially, we can represent the morphism $\mathbb{O}(f)$ as a forest as in \cref{fig:DualForest}, thinking leaves $j\in \mathbb{O}(T)$ as being attached to the root $k\in \mathbb{O}(S)$ if $\mathbb{O}(f)(j)=k$. 
	\begin{figure}[htb]
		\begin{center}
			\begin{tikzpicture}
			\foreach \x/\lab/\nom in {0.5/a_0/a0,1.5/a_1/a1,2.5/a_2/a2,3.5/a_3/a3}{
				\path (0,-\x) node (\nom) {$\lab$}; 
			};
			\begin{scope}[yshift=12]
			\foreach \x/\lab/\nom in {0.5/b_0/b0,1.5/b_1/b1,2.5/b_2/b2,3.5/b_3/b3,4.5/b_4/b4}{
				\path (4,-\x) node (\nom) {$\lab$}; 
			};
			\end{scope}
			\draw[->] (a0) to (b0);
			\draw[->] (a1) to (b2);
			\draw[->] (a2) to (b3);
			\draw[->] (a3) to (b3);
			\end{tikzpicture}
			\qquad
			\begin{tikzpicture}
			\foreach \x/\lab/\nom in {0.5/a_0/a0,1.5/a_1/a1,2.5/a_2/a2,3.5/a_3/a3}{
				\path (0,-\x) node (\nom) {$\lab$}; 
			};
			\foreach \y in {0,1,2,3,4}{
				\draw[blue](-0.2,-\y) to (0.2,-\y); 
			};
			\path (2,-4) node (blue1){};
			\path[fill=blue] (blue1.center) circle (0.05); 
			\path (2,-2.3) node (blue2){};
			\path[fill=blue] (blue2.center) circle (0.05); 
			\path (2,-1) node (blue3){};
			\path[fill=blue] (blue3.center) circle (0.05); 
			\path (2,0.3) node (blue4) {};
			\path[fill=blue] (blue4.center) circle (0.05); 
			
			\begin{scope}[yshift=12]
			\foreach \x/\lab/\nom in {0.5/b_0/b0,1.5/b_1/b1,2.5/b_2/b2,3.5/b_3/b3,4.5/b_4/b4}{
				\path (4,-\x) node (\nom) {$\lab$}; 
			};
			\foreach \y in {0,1,2,3,4,5}{
				\draw[blue](4.2,-\y) to (3.8,-\y); 
			}
			\end{scope}
			\draw[->] (a0) to (b0);
			\draw[->] (a1) to (b2);
			\draw[->] (a2) to (b3);
			\draw[->] (a3) to (b3);
			\draw[blue] (3.8,0.42) to (blue4.center);
			\draw[->,blue] (blue4.center) to (0.2,0);
			\draw[blue] (3.8,-0.58) to (blue3.center);
			\draw[blue] (3.8,-1.58) to (blue3.center);
			\draw[->,blue] (blue3.center) to (0.2,-1);
			\draw[blue] (3.8,-2.58) to (blue2.center);
			\draw[->,blue] (blue2.center) to (0.2,-2);
			\draw[blue] (3.8,-3.58) to (blue1.center);
			\draw[blue] (3.8,-4.58) to (blue1.center);
			\draw[->,blue] (blue1.center) to (0.2,-4);
			\end{tikzpicture}
		\end{center}
		\caption{Left: a morphism $f$ of linearly ordered sets. Right: the morphism $\mathbb{O}(f)$, visualized as a forest (blue).}\label{fig:DualForest}
	\end{figure}
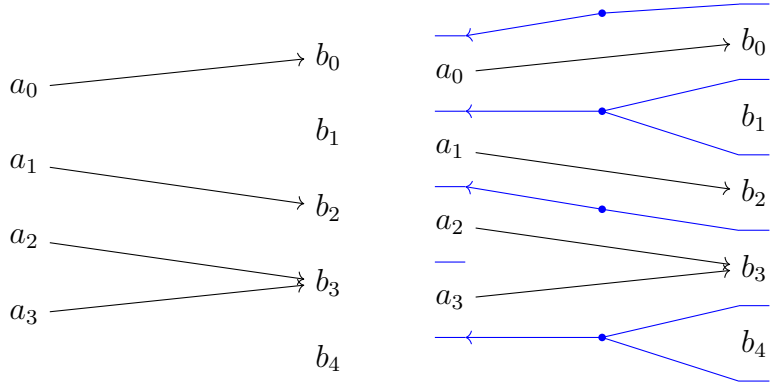  
	
	Note that the functors $\mathbb{I}$ and $\mathbb{O}$ define an equivalence of categories. Since $\Delta_+$ (resp. $\nabla_+$) is the skeletal version of $\bbDelta_+$ (resp. $\bbNabla_+$), all isomorphisms in these categories are identities, we see that we get an induced \emph{isomorphism} of categories
	\[
	O:\Delta_+^\op\overset{\cong}{\longleftrightarrow }\nabla_+:I
	\] 
	Moreover, we can define a functor $\bbNabla_+\to \Fin_\ast$ by 
	\[
	S\mapsto (S\amalg \{\ast\})_{/\on{max}(S)\sim\on{min}(s)\sim \ast}
	\]
	We then find that the induced functor 
	\[
	\Delta_+^{\op}\hookrightarrow \Delta_+^\op\overset{O}{\to} \nabla_+\to \Fin_\ast
	\]
	is precisely the functor $\on{cut}:\Delta^\op\to \Fin_\ast$ defined in \cite[4.1.2.9]{LurieHA}.
\end{const}

\begin{defn}
	Given two linearly ordered sets $S,T\in \bbDelta_+$ define the \emph{ordinal sum} $S\oplus T$ to be the set $S\amalg T$ equipped with the linear order defined by the orders on $S$ and $T$ and the proscription that for all $s\in S$ and $t\in T$, $s\leq t$. The ordinal sum defines a monoidal structure on $\bbDelta_+$. 
	
	Given two linearly ordered sets $S,T\in \bbNabla_+$, with $b$ the maximum of $S$ and $a$ the minimum of $T$, define the \emph{imbrication} $S\star T$ to be the linearly ordered set $(S\oplus T)_{/a\sim b}$ (note that since $a$ is the successor to $b$ in $S\oplus T$, there is a canonical linear order on $S\star T$ compatible with the quotient map). 
\end{defn}

\begin{lem}
	The functor $\mathbb{O}$ is a monoidal functor sending the ordinal sum to the imbrication. 
\end{lem}

\begin{defn}
A \emph{cyclic order} on a finite set $S$ is a transitive $\mathbb{Z}$-action on $S$. Equivalently, this is simply transitive action of $\mathbb{Z}/|S|$ on $S$. 
\end{defn}

\begin{defn}
	Given a cyclic set $S$, and a collection $\{[n_i]\}_{i\in S}$ of objects in $\Delta_+$, we define a cyclic set $\bigcup^S([n_i])$ as follows. The underlying set is $\coprod_{i\in S} [n_i]$, and the cyclic order is given by the $\mathbb{Z}/n$-action (where $n:=\sum_{i\in S} (n_i+1)$) that sends 
	\[
	j\in[n_i]\mapsto \begin{cases}
	j+1 & j<n_i\\
	0\in [n_{i+1}] & j=n_i
	\end{cases}
	\] 
	where $i+1$ denotes the successor of $i$ in the cyclic order on $S$. We call this order on $\coprod_{i\in S} [n_i]$ the \emph{lexicographic (cyclic) order}.
\end{defn}

\begin{defn}\label{defn:CyclicCats}
A \emph{morphism of cyclically ordered sets} $S\to T$  consists of a map of sets $\phi:S\to T$, and a linear order on each fiber such that the lexicographic cyclic order on $S$ agrees with the predefined cyclic order on $S$.
	
The \emph{cyclic category} has as its objects the standard cyclicly ordered sets $\langle n\rangle$ for $n\geq 0$, and as its morphisms the maps of finite sets respecting the cyclic order. The \emph{enlarged cyclic category} $\bbLambda$ has as its objects all finite, non-empty, cyclically ordered sets, and as its morphisms the maps which respect the cyclic order.
\end{defn}

\begin{const}[Cyclic Duality]\label{const:CyclicDuality}
In analogy to the construction of the linear interstice functors, we define a duality  
\[
\mathbb{D}:\bbLambda^{\op}\to \bbLambda
\]
on the cyclic category. Let $S\in \bbLambda$ be a cyclicly ordered set. We define a \emph{cyclic interstice} of $S$ to be an ordered pair $(a,a+1)\in S\times S$, where $a+1$ denotes the successor of $a$ under the cyclic order. We denote the set of cyclic interstices of $S$ by $\mathbb{D}(S)$. The set $\mathbb{D}(S)$ inherits a canonical cyclic order from $S$, which can be visualized as in \cref{fig:CyclicInterstices}. 
\begin{figure}[htb]
\begin{tikzpicture}
\draw (0,0) circle (1);
\foreach \th in {0,60,120,180,240,300}{
\draw[fill=black] (\th:1) circle (0.05);
};
\foreach \th in {30,90,150,210,270,330}{
\path[blue] (\th:1) node {$\times$};
};
\end{tikzpicture}
\caption{A cyclic set with its cyclic order visualized via an embedding into the oriented circle (black), together with its set of cyclic interstices (blue crosses).}\label{fig:CyclicInterstices}
\end{figure}
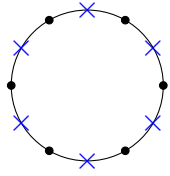
The functor $\mathbb{D}$ is specified on morphisms by an analogue of \cref{const:LinInterstice}, namely, for $f:S\to T$ in $\bbLambda$, we set%\todo{Give explicit characterization}
\[
\mathbb{D}(f)(j,j+1):=k\quad \text{where }  
\]
This functor is an equivalence of categories.%\todo{Cite DK}.
Since $\Lambda$ is the skeletal version of $\bbLambda$, $\mathbb{D}$ descends to an equivalence $D:\Lambda^\op\to \Lambda$
\end{const}

\begin{const}[Cyclic closures]\label{const:CyclicClosures}
We define a functor $\mathbb{K}:\bbDelta\to \bbLambda$ in the following way. Given a linearly ordered set $S$ of cardinality $n+1$, there is a unique order-preserving bijection $\phi:S\to [n]$. We define a bijection 
\[
S\to r(n); \quad j\mapsto \exp\left(\frac{2\pi i \phi(j)}{n+1}\right) 
\]
to the $n^{\on{th}}$ roots of unity in $S^1$. The orientation on $S^1$ then yields a canonical cyclic order on $S$. Passing to skeletal versions yields the well-known functor $\kappa: \Delta\to \Lambda$. 

Via the equivalences $\mathbb{O}$ and $\mathbb{D}$ we can then define a functor $\mathbb{C}:\bbNabla\to \Lambda$ such that the diagram 
\[
\begin{tikzcd}
\bbDelta^\op\arrow[d,"\mathbb{K}"']\arrow[r,"\mathbb{O}"] & \bbNabla\arrow[d,"\mathbb{C}"]\\
\bbLambda^\op\arrow[r,"\mathbb{D}"'] & \bbLambda
\end{tikzcd}
\]
commutes up to natural isomorphism. The functor $\mathbb{C}$ admits the following explicit description on objects. Let $S\in \bbNabla$ with maximal element $b$ and minimal element $a$. Then $\mathbb{C}(S)$ can be identified with with quotient of $\mathbb{K}(S)$ by the identification $a\sim b$. Once again, we have that $\mathbb{C}$ descends to a functor $C:\nabla\to \Lambda$. 
\end{const}

\begin{defn}\label{defn:CyclicCatChosenTriv}
Given an object $S\in \bbLambda$, a \emph{linear order on $S$ compatible with the cyclic order} consists of a pair $([n],\phi)$ consisting of an object $[n]\in \Delta$, and an isomorphism $\phi:\mathbb{K}([n])\cong S$. 

We introduce one more equivalent variant of $\bbLambda$, which we will denote $\bfLambda$. The objects of $\bfLambda$ consist of pairs $(S,\phi)$ where $S\in \bbLambda$, and $\phi:\mathbb{K}([n])\cong S$  is a compatible linear order on $S$. The morphisms of $S$ are simply the morphisms of $\bbLambda$.  It is clear that the forgetful functor $\bfLambda\to\bbLambda$ is an equivalence. 
\end{defn}

\begin{const} 
The functor $\mathbb{K}$ clearly extends to a functor $\mathbf{K}: \bbDelta \to \bfLambda$ by choosing the identity as the compatible linear order. We can then define functors $\mathbf{D}: \bfLambda^\op\to \bfLambda$ and $\mathbf{C}:\bbNabla\to \bfLambda$ such that the diagram 
\[
\begin{tikzcd}
\bbDelta^\op\arrow[d,"\mathbb{K}"']\arrow[r,"\mathbf{O}"] & \bbNabla\arrow[d,"\mathbf{C}"]\\
\bfLambda^\op\arrow[r,"\mathbf{D}"'] & \bfLambda
\end{tikzcd}
\]
commutes \emph{strictly}. %\todo{fill in details}
\end{const}

\begin{lem}\label{lem:LinAndCycOrdSum}
Let $S\in \bbLambda$, a set $\{[n_i]\}_{i\in S}$ of elements in $\Delta_+$, and a compatible linear order $\phi:\mathbb{K}([m])\cong S$, there is a canonical isomorphism 
\[
\mathbb{K}\left(\bigoplus_{i\in [m]}[n_{\phi(i)}]\right)\cong \bigcup\nolimits^S[n_i] 
\] 
which acts as the identity on underlying sets. 
\end{lem}

\begin{proof}
We compare the $\mathbb{Z}/n$-actions. When $j\in \bigoplus_{i\in[m]} [n_i]$ is not maximal, the successor function for the ordinal sum agrees with the $\mathbb{Z}/n$-action on $\bigcup^S[n_i]$. If $j$ is maximal, we have that the action on the left sends $j$ to $0\in n_{\phi(0)}$, which agrees with the definition of the cyclic order on the right. 
\end{proof}
\subsection{Calabi-Yau algebras}
Throughout the following section, we take $\CC^\otimes\to \bFin_\ast$ to be a symmetric monoidal $\infty$-category with monoidal unit $\bbOne$ and tensor product $\otimes$.

\begin{const}
There is a functor $B:\bbLambda\to\Ass$ defined as follows. On objects, send each $S\in\bbLambda$ to $S\amalg \{\ast\}$, forgetting the cyclic order. On morphisms, send $f:S\to T$ to its underlying map of sets. Define a linear order on the fibers of $f$ by choosing embeddings of $S$ and $T$ into $S^1$ compatible with the cyclic order, and representing $f$ as a commutative diagram 
\[
\begin{tikzcd}
S^1\arrow[r, "\tilde{f}"] & S^1 \\
S\arrow[u,"\alpha"]\arrow[r, "f"] & T\arrow[u,"\beta"']
\end{tikzcd}
\] 
where $\tilde{f}$ is monotone of degree 1. For $i\in T$, the preimage of $\beta(i)$ under $\tilde{f}$ is an interval, and $\beta(f^{-1}(i))\subset \tilde{f}^{-1}(\beta(i))$. The orientation of $S^1$ induces an orientation of $\tilde{f}^{-1}(\beta(i))$, and hence a linear order on $f^{-1}(i)$. 
\end{const}

\begin{defn}
The \emph{cyclic bar object} of an algebra object $X:\Ass\to \CC^\otimes$ is the composition $B^\ast(X)$. A \emph{cyclic trace} on $X$ is a natural transformation $\eta$ from $B^\ast(X)$ to the constant cyclic object on $\bbOne\in\CC$. We call a pair $(X,\eta)$ consisting of an algebra object in $\CC^\otimes$ and a cyclic trace a \emph{trace algebra}. 
\end{defn}

\begin{rmk}
A natural transformation to a constant cyclic object may be modeled as a functor from the category $\bbLambda_\diamond$ obtained from $\bbLambda$ by formally adjoining a terminal object. We denote the terminal object of $\bbLambda_\diamond$ by $\diamond$. 
\end{rmk}

\begin{defn}\label{defn:nondegenerate}
A morphism $\gamma:X\otimes X\to \bbOne$ in $\CC$ is called \emph{non-degenerate} if there exists a morphism $\eta:\bbOne\to X\otimes X$ such that 
\begin{itemize}
\item The composite
\[
X \overset{\simeq}{\to} X\otimes \bbOne \overset{\eta\otimes\id_{\bbOne}}{\longrightarrow} X\otimes X\otimes X \overset{\id_{\bbOne}\otimes\gamma}{\longrightarrow} \bbOne\otimes X\overset{\simeq}{\to} X
\]
is homotopic to the identity.
\item The composite 
\[
X \overset{\simeq}{\to} \bbOne\otimes X \overset{\id_{\bbOne}\otimes\eta}{\longrightarrow} X\otimes X\otimes X \overset{\gamma\otimes\id_{\bbOne}}{\longrightarrow} X\otimes \bbOne\overset{\simeq}{\to} X
\]
is homotopic to the identity.
\end{itemize}
\end{defn}

\begin{defn}
Let $(X,\eta)$ be a trace algebra in $\CC$, and let $\eta_2:X\otimes X\to \bbOne$ be the map induced by $\langle 2\rangle \to \diamond$ in $\bbLambda_\diamond$ under $\eta$. We call $(X,\eta)$ a \emph{Calabi-Yau algebra} in $\CC$ if $\eta_2$ is non-degenerate.
\end{defn}

\begin{rmk}
The definition above is precisely that of \cite[Example 4.2.8]{LurieTFT}. When Hochschild homology is defined, the map $\eta:B^\ast(X)\to \bbOne$ is equivalently an $S^1$-equivariant trace 
\[
\int_{S^1} X\to \bbOne.
\]
\end{rmk}

\begin{defn}
Let $\Ass_\CY$ be the category with
\begin{itemize}
\item Objects $\on{ob}(\Ass)\amalg \{\diamond\}$.
\item Morphisms between $S,T\in \Ass$ 
\[
\Hom_{\Ass_\CY}(S,T):=\Hom_{\Ass}(S,T).
\]
\item For $S\in \Ass$,
\[
\Hom_{\Ass_\CY}(\diamond,S):= \emptyset
\]
and a morphism $S\to \diamond$ is a choice of a subset $T\subset S^\circ$ and a cyclic order on $T$. 
\item For $S,T\in \Ass$, and morphisms $\phi:S\to T$ and $\psi:T\to \diamond$, the composite $\psi\circ \phi$ is given by the induced cyclic order
\end{itemize}
Note that $\Ass_\CY$ comes equipped with a functor $\Ass_\CY\to \bFin_\ast$ sending $\diamond\mapsto \langle 1\rangle$. 
\end{defn}

\begin{const}
Let $\bbLambda\to \bbLambda_\diamond$ and $\Ass\to \Ass_\CY$ be the inclusions. Define a functor $F: \bbLambda_\diamond\to \Ass_\CY$ by setting $F=B$ on $\bbLambda\subset \bbLambda_\diamond$, and sending $\diamond\mapsto \diamond$.  By definition, the diagram
\begin{equation}\label{eq:AssCYDiagram}
\begin{tikzcd}
\bbLambda\arrow[r]\arrow[d, "B"'] & \bbLambda_\diamond\arrow[d,"F"]\\
\Ass\arrow[r] & \Ass_CY
\end{tikzcd}
\end{equation}
commutes.
\end{const}

\begin{defn}
	We take $\mathfrak{P}$ to be the categorical pattern of \cite[Proposition 2.1.4.6]{LurieHA}. In the following proof, we will freely make reference to this proposition, and Appendix B from the same. 
\end{defn}

\begin{lem}
The diagram 
\[
\begin{tikzcd}
N(\bbLambda)\arrow[r]\arrow[d, "B"'] & N(\bbLambda_\diamond)\arrow[d,"F"]\\
N(\Ass)\arrow[r] & N(\Ass_CY)
\end{tikzcd}
\]
induces an $\mathfrak{P}$-anodyne morphism of $\infty$-categories  
\[
\theta:N(\Ass)\coprod_{N(\bbLambda)}N(\bbLambda_\diamond)\to N(\Ass_{\CY}) 
\]
over $\bFin_\ast$, where the non-degenerate marked simplices are precisely the inert morphisms of $\Ass$. 
\end{lem}
\begin{proof}
An $n$-simplex of $N(\Ass)\coprod_{N(\bbLambda)}N(\bbLambda_\diamond)$ is an equivalence class in $N(\Ass)\amalg N(\bbLambda_\diamond)$ under the relation that 
\[
\underbrace{(S_0\to S_1\to\cdots \to S_n)}_{\in N(\Ass)_n}\sim \underbrace{(T_0\to T_1\to\cdots \to T_n)}_{\in N(\bbLambda)_n}
\]
if and only if
\[
B(T_0\to T_1\to\cdots \to T_n)=(S_0\to S_1\to\cdots \to S_n).
\]
In particular, $\theta$ is injective, and a bijection on 0-simplices.

We proceed by induction. For ease of notation, we set $Q=N(\Ass)\coprod_{N(\bbLambda)}N(\bbLambda_\diamond)$.
\begin{enumerate}
\item Suppose $f:S\to \diamond$ is a 1-simplex not contained in the image of $\theta$. Then $S$ is determined by $T\subsetneq S^\circ $ and a cyclic order on $S$.  Adding a basepoint to $T$ to get $T_f\in \Ass_\CY$ we get a factorization of $f$ as 
\[
\begin{tikzcd}
 & T_f\arrow[dr,"\alpha"] & \\
 S\arrow[rr,"f"]\arrow[ur, "\beta"] & & \diamond
\end{tikzcd}
\]
in $\Ass_\CY$. Taking such a 2-simplex $\sigma_f$ for every such $f$, we can form the pushout 
\[
\begin{tikzcd}
\coprod_{\{f\}}(\Lambda^2_1)^\flat\arrow[r]\arrow[d] & Q_0\arrow[d]\\
\coprod_{\{f\}} (\Delta^2)^\flat\arrow[r] & Q_1
\end{tikzcd}
\]
The morphism on the left is of type ($C_1$) from \cite[B.1.1]{LurieHA}, so we get a factorization 
\[
\begin{tikzcd}
Q_0\arrow[r,"\tau_1"]\arrow[rr, out=30,in=150, "\theta"] & Q_1\arrow[r,"\theta_1"] & N(\Ass_\CY)
\end{tikzcd}
\]
where $\tau_1$ is $\mathfrak{P}$-anodyne, and $\theta_1$ is bijective on $1$-simplices.
\item Now suppose that $\sigma:\Delta^2\to \Ass_\CY$ is a 2-simplex not in the image of $\theta_1$. Then $\sigma$ must be given by a sequence 
\[
\begin{tikzcd}
S_1\arrow[r, "g"] & S_2\arrow[r,"f"] & \diamond
\end{tikzcd}
\]
(if $\sigma$ does not contain $\diamond$, it is the image of a simplex in $\Ass$, if it contains two copies of $\diamond$, it is degenerate). Consequently, we get two 2-simplices, $\sigma_{f\circ g}$ and $\sigma_g$ in the image of $\theta_1$. Moreover, $g$ restricts to a morphism 
\[
g:T_{f\circ g}\to T_f,
\]
and we get a 2-simplex $S_1\to S_2 \to T_f$. We then note that the $\Lambda^3_1$ horn 
\[
\begin{tikzcd} 
  & T_{f}\arrow[ddr]& \\
  & T_{f\circ g}\arrow[u]\arrow[dr]& \\
  S_1\arrow[uur]\arrow[rr]\arrow[ur]& &\diamond
\end{tikzcd}
\]
can be filled to a 2-simplex $S_1\to T_f\to \diamond$ via a horn of type ($C_1$). Finally, we get a $\Lambda^3_2$-horn 
\[
\begin{tikzcd} 
  & S_2\arrow[ddr,"f"]\arrow[d]& \\
  & T_{f}\arrow[dr]& \\
  S_1\arrow[uur,"g"]\arrow[rr,"f\circ g"']\arrow[ur]& &\diamond
\end{tikzcd}
\]
of type ($C_1$). This gives us a factorization of $\theta$ as $Q_0\overset{\tau_2}{\to} Q_2\overset{\theta_2}{\to} N(\Ass_\CY)$
where $\tau_1$ is $\mathfrak{P}$-anodyne and $\theta_2$ is bijective on simplices of dimension $\leq 2$.  
\item Now suppose inductively that we have obtained a factorization through $\theta_{n-1}:Q_n\to N(\Ass_\CY)$ such that
\begin{itemize}
\item $\theta_{n-1}$ is bijective on $k$-simplices for $k\leq n-1$.
\item The image of $\theta_{n-1}$ contains all $n$-simplices of the form 
\[
S_0\to S_1\to\cdots \to S_{n-1}\to \diamond
\]
where $S_{n-1}\to \diamond$ is a 1-simplex in the image of $\Lambda_\diamond$. 
\end{itemize}
Suppose given an $n$-simplex $\sigma$ not in the image of $\theta_{n-1}$. Then, by similar reasoning to that above, $\sigma$ must be of the form
\[
S_0\overset{\phi_1}{\to} S_1\overset{\phi_2}{\to}\cdots \to S_{n-1}\overset{\phi_n}{\to}\diamond
\]
with $S_{n+1}\to \diamond$ not in the image of $\Lambda_\diamond$. Define $\psi_k:=\phi_n\circ\phi_{n-1}\circ\cdots\circ\phi_{n-k}$, we then get $n$-simplices in the image of $\theta_{n-1}$ 
\[
S_0\overset{\phi_1}{\to} S_1\overset{\phi_2}{\to}\cdots\to \widehat{S_k} \to S_{n-1}\to T_{\phi_n}{\to}\diamond
\]
and an $n$-simplex in the image of $\theta_{n-1}$
\[
S_0\to S_1\to \cdots \to S_{n-1}\to T_f.
\]
These $n$ $n$-simplices form a $\Lambda^{n+1}_n$-horn in $N(\Ass_\CY)$ which, once again, can be filled by a pushout of type ($C_1$).
\end{enumerate}
We therefore get a factorization
\[
Q_0\to Q_1\to \cdots \to N(\Ass_\CY)
\]
which exhausts $N(\Ass_\CY)$. Each morphism in this sequence is $\mathfrak{P}$-anodyne, and so the transfinite composition $Q_0\to N(\Ass_\CY)$ is $\mathfrak{P}$-anodyne.
\end{proof}

\begin{cor}
The $\infty$-category of trace algebras in $\CC$ is equivalent to the full subcategory of $\on{Map}^\sharp_{\bFin_\ast}(N(\Ass_\CY),\CC^\otimes)$ sending $\diamond$ to $\bbOne$.  
\end{cor}

\begin{defn}
We define the \emph{$\infty$-category of Calabi-Yau algebras in $\CC$} to be the full subcategory of $\on{Map}^\sharp_{\bFin_\ast}(N(\Ass_\CY),\CC^\otimes)$ on those objects which 
\begin{enumerate}
\item send $\diamond$ to $\bbOne$, and 
\item send the morphism $\langle 2\rangle \to \diamond$ in $\bbLambda_\diamond$ to a non-degenerate morphism $X\times X\to \bbOne$. 
\end{enumerate}
\end{defn}

\subsection{Cartesian monoidal structures}

Throughout this paper, we will model (symmetric) monoidal structions by Cartesian fibrations, rather than the coCartesian fibrations used in \cite{LurieHA}. These fibrations will be defined via adjunctions with the following. Throughout this section, $\CC$ will denote an $\infty$-category which admits finite products. 

\begin{defn}
The category $\Delta^{\amalg}$ has as its objects  pairs $([n],\{i,j\})$, where $[n]\in \Delta$ and $i\leq j$ are elements in $[n]$. The morphisms $ ([n],\{i,j\})\to ([m],\{k,\ell\})$ consist of a morphism $\phi:[n]\to [m]$ such that $\phi(i)\leq k\leq \ell\leq \phi(j)$.  We will, in general, think of $\{i,j\}$ as an interval inside $[n]$, and denote by $\{i\leq j\}$ the linearly ordered set 
\[
\{i\leq j\}:= \{i,i+1,\ldots,j\}\subset [n].
\]

The category $\bFin_\ast^\amalg$ has as its objects pairs $(S, T)$ where $S\in \bFin_\ast$ and $T\subset S^\circ$. A morphism $(S,T)\to (P,Q)$ consists of a morphism $\phi:S\to P$ in $\bFin_\ast$ such that $\phi(T)\subset Q$. We will sometimes denote by $\bbGamma^\amalg$ the category $(\bFin_\ast^\amalg)^\op$.   
\end{defn}

\begin{rmk}
We can provide an alternate characterization of $\Delta^\amalg$ and $\bFin_\ast^\amalg$. The functor $\Delta^\amalg\to \Delta$  is the coCartesian fibration defined as a Grothendieck construction of the functors
\[
\Delta\to \Cat;\quad [n] \mapsto I_{[n]}^\op.
\]
The functor $\bFin_\ast^\amalg\to \bFin$ is the Cartesian fibration defined as a Grothendieck construction of the (contravariant) power set functor 
\[
\bFin_\ast^{\op}\to \Cat;\quad S\mapsto \mathcal{P}(S^\circ).
\]

Note that, as in \cite[Remark 10.3.2]{DKHSSI}, these constructions relate to the constructions $\Delta^\times \to \Delta$ and $\Gamma^\times\to \Fin_\ast$ from \cite[Proposition 1.2.8]{LurieDAGII} and \cite[Proposition 2.4.1.5]{LurieHA} respectively. In particular, the functor $\Gamma^\times\to \Fin_\ast$ is the Cartesian fibration arising as the Grothendieck construction of 
\[
\bFin^\ast \to \Cat;\quad S\mapsto \mathcal{P}(S^\circ)^\op.
\]
For an $\infty$-category $\DD$ with enough colimits, the functor $\bFin_\ast^\amalg\to \bFin_\ast$ can therefore be used to construct a coCartesian fibration $\DD^\amalg\to \bFin_\ast$ modeling the coCartesian symmetric monoidal structure on $\DD$. 
\end{rmk}

\begin{const}
The functor $\on{cut}:\Delta\to \bFin_\ast^\op$ yields a functor $\Delta^\amalg\to (\bFin_\ast^\amalg)^\op$. To see this, we first note that for $\{i,j\}\subset [n]$ in $\Delta^\amalg$, we have $\mathbb{O}(\{i\leq j\})\subset \mathbb{O}([n])$. On objects we therefore define $\{i,j\}\subset [n]\mapsto (\mathbb{O}([n]),\mathbb{O}(\{i\leq j\}))$ 

Given a morphism $f:([n],\{i,j\})$ to $([m],\{k,\ell\})$ in $\Delta^\amalg$, we get a morphism $\mathbb{O}(f): \mathbb{O}([m])\to \mathbb{O}([n])$. Moreover, the condition that $f(i)\leq k\leq \ell\leq f(j)$ ensures that $\mathbb{O}(f)\left( \mathbb{O}(\{k\leq \ell\})\right) \subset \mathbb{O}(\{i\leq j\})$. 
\end{const}

\begin{const}[Cartesian monoidal structures]
Given an $\infty$-category $\CC$ with finite products, we can associate two Cartesian fibrations to $\CC$ as follows. 

We define a functor of $\infty$-categories $\overline{\CC^\boxtimes} \to \Delta$ via the universal property 
\[
\Hom_\Delta(K, \overline{\CC^\boxtimes})\cong \Hom_{\Set_\Delta} (K\times_\Delta \Delta^\amalg, \CC).
\]
Similarly, we define a functor $\overline{\CC^\times}\to \bbGamma$ via the universal property 
\[
\Hom_{\bbGamma}(K, \overline{\CC^\boxtimes})\cong \Hom_{\Set_\Delta} (K\times_{\bbGamma} \bbGamma^\amalg, \CC).
\]
Both of these are Cartesian fibrations by dint of \cite[3.2.2.13]{LurieHTT}.

We now let $\CC^\boxtimes\subset \overline{\CC^\boxtimes}$ be the full subcategory on those objects $G:I_{[n]}^\op\to \CC$ for which $G$ displays $G(\{i\leq j\})$ as a product over $G(\{k\leq k+1\})$ for $i\leq k<j$.

Similarly, we let $\CC^\times\subset \overline{\CC^\times}$ be the full subcategory on those objects $G:\mathcal{P}(S^\circ)^{\op}\to \CC$ for which  $G$ displays $G(S)$ as a product over $G(i)$ for $i\in S$.
\end{const}

\begin{prop}
The functor $\CC^\boxtimes\to \Delta$ is a Cartesian fibration exhibiting the Cartesian monoidal structure on $\CC$.
\end{prop}

\begin{proof}
This is \cite[Prop. 10.3.8]{DKHSSI}.
\end{proof}

\begin{prop}
The functor $\CC^\times\to \bbGamma$ is a Cartesian fibration exhibiting the Cartesian symmetric monoidal structure on $\CC$. 
\end{prop}

\begin{proof}
The proof of this statement is, {\itshape mutatis mutandis}, the same as the proof of \cite[Proposition 2.4.1.5]{LurieHA}. 
\end{proof}
\subsection{$\infty$-Categories of Spans}
We will briefly recall here the requisite constructions and definitions for $\infty$-categories of spans. For a fuller exposition, see \cite[Chapter 10]{DKHSSI}. Throughout this section, we will assume that $\CC$ is now an $\infty$-category with small limits. 

\begin{defn}
	Let $S$ be a linearly ordered set. We define $I_{S}$ to be the poset of non-empty sub-intervals $\{i\leq j\}\subset S$. 
	
	Let $\Delta^n$ be the standard $n$-simplex. We define the \emph{spine} $\EuScript{J}^n\subset\Delta^n$ to be
	\[
	\EuScript{J}^n:= \Delta^{\{0,1\}}\coprod_{\Delta^{\{1\}}} \Delta^{\{1,2\}}\cdots \coprod_{\Delta^{\{n-1\}}} \Delta^{\{n-1,n\}}.
	\] 
\end{defn}

\begin{const}[Categories of Spans]\label{const:TwSpan}
We define the functor $\Tw:\Delta\to \Set_\Delta$ by 
\[
[n]\mapsto N(I_{[n]})^\op.
\]
By left Kan extension along the Yoneda embedding and restriction, we get an adjunction, which we will also denote by
\begin{equation}\label{eq:TwSpanKanAdjunction}
\Tw:\Set_\Delta\leftrightarrow \Set_\Delta: \overline{\Span}.
\end{equation}
For an $\infty$-category $\DD$, the simplicial set $\Tw(\DD)$ is an $\infty$-category, which we will call the \emph{twisted arrow $\infty$-category} of $\DD$. Note that $\Tw(\DD)$ comes with a canonical projection $\eta_\DD:\Tw(\DD)\to \DD$. If $\DD$ is the nerve of a 1-category $D$, $\Tw(\DD)$ can be identified with the nerve of the 1-category $\Tw(D)$ whose objects are morphisms $f:a\to b$ in $\DD$ and whose morphisms $f\to g$ are commutative diagrams  
\[
 \begin{tikzcd}
 a\arrow[r, "f"]\arrow[d] & b \\
 c\arrow[r, "g"'] & d\arrow[u]
 \end{tikzcd}
\]
in $D$, i.e. factorizations $f=h\circ g\circ \ell$. 

Given $X\in \Set_\Delta$, we can extend the adjunction \ref{eq:TwSpanKanAdjunction} to an adjunction 
\[
\Tw_X:(\Set_\Delta)_{/X} \leftrightarrow (\Set_\Delta)_{/X}: \overline{\Span}_X
\]
by setting $\Tw_X(S\to X)$ to be the composite 
\[
\Tw(S)\to \Tw(X)\overset{\eta_X}{\to} X
\]
and by setting $\overline{Span}_X(S\to X)$ to be the left-hand column of the pullback
\[
\begin{tikzcd}
\overline{\Span}_X(S)\arrow[r]\arrow[d] & \overline{\Span}(S)\arrow[d]\\
X\arrow[r] & \overline{\Span}(X)
\end{tikzcd}
\]
in $\Set_\Delta$.

Let $p:S\to X$ be a map of simplicial sets. We call an $n$-simplex in $\overline{\Span}_X(S)$ represented by a map $\sigma:\Tw(\Delta^n)\to S$ a \emph{Segal simplex} if, for every $\Delta^k\subset \Delta^n$, the composite diagram 
\[
\{0,k\}\star\Tw(\EuScript{J}^k)\subset \Tw(\Delta^k)\subset \Tw(\Delta^n)\overset{\sigma}{\to} S
\]
is a $p$-limit diagram. We denote by $\Span_X(S)\subset \overline{\Span}_X(S)$ the simplicial subset consisting of the Segal simplices. 
\end{const}

\begin{prop}[{\cite[10.2.31]{DKHSSI}}]
	Let $p:\CC^\otimes\to N(\Delta)$ be a Cartesian fibration exhibiting a monoidal structure on $\CC^\otimes_[1]$ such that $p$ admits relative pullbacks. Then $\Span_\Delta(\CC^\otimes)\to N(\Delta)$ is a Cartesian fibration exhibiting a monoidal structure on $\Span_\ast(\CC^\otimes_{[1]})$.    
\end{prop}

\begin{cor}
	Let $p:\CC^\otimes \to N(\Gamma)$ be a Cartesian fibration exhibiting a symmetric monoidal structure on $\CC^\otimes_{\langle 1\rangle}$ such that $p$ admits relative pullbacks. Then $\Span_\Gamma(\CC^\otimes)\to N(\Gamma)$ is a Cartesian fibration exhibiting a symmetric monoidal structure on $\Span_\ast(\CC^\otimes)$.
\end{cor}

\begin{cor}\label{cor:MonoidalCatsOfSpans}
	Let $\CC$ be an $\infty$-category that admits small limits. Then the functors 
	\begin{align*}
	\Span_\Delta(\CC^\boxtimes) & \to N(\Delta)\\
	\Span_{\bbGamma}(\CC^\times) & \to N(\bbGamma)
	\end{align*}
	are Cartesian fibrations exhibiting a monoidal or a symmetric monoidal structure on $\Span_\ast(\CC)$ respectively. 
\end{cor}

\begin{rmk}
	The monoidal structures from \cref{cor:MonoidalCatsOfSpans} can be seen as `pointwise cartesian' monoidal structure, with monoidal product given by the product in $\CC$. 
\end{rmk}

\section{Algebras in Spans}
Throughout this section, we set $\Theta:=\Tw(\Delta)\times_\Delta \Delta^\amalg$. Morphisms in $\Theta$ will be represented as diagrams 
\[
\begin{tikzcd}
 \{i,j\}\arrow[Subseteq]{r}{}& {{[n]}}\arrow[d, "g"']\arrow[r, "f"] & {{[m]}} \\
 \{i^\prime,j^\prime\}\arrow[Subseteq]{r}{}& {[n^\prime]}\arrow[r,"f^\prime"'] & {[m^\prime]}\arrow[u, "\overline{g}"']
\end{tikzcd}
\]
in $\Delta$. In this section and the next, $\CC$ will denote an $\infty$-category with small limits. We will, on occasion, denote an object $\{i,j\}\subset[n]\overset{f}{\to}[m]$ in $\Theta$ by the pair $(f, \{i,j\})$.

\subsection{Conditions on functors}
Suppose we are given a functor $G: \Theta\to  \CC$, which corresponds to a functor
\[
\tilde{G}: \Tw(\Delta)\to \CC^{\boxtimes}
\]
over $\Delta$. 

\begin{prop}\label{prop:Deltapullbackcond}
The functor $G$ defines a functor 
$\overline{G}: \Delta\to \Span_\Delta (\CC^{\boxtimes})$
if and only if, for every simplex
$[n_0] \overset{\phi_1}{\to}[n_1]\overset{\phi_2}{\to}\cdots \overset{\phi_k}{\to} [n_k]$
in $\Delta$ and every interval $\{i,j\}\subset [n_0]$, the corresponding diagram 
\begin{equation}\label{diag:segalsimpsinHSpan}
\begin{tikzcd}[column sep=-40pt]
 & & & G(\phi_n\circ \cdots \circ \phi_1,\{i,j\})\arrow[dll]\arrow[drr] &[-20pt] &[-30pt] &[-20pt] \\
 & G(\phi_1,\{i,j\})\arrow[dr]\arrow[dl] & &  \cdots\arrow[dr]\arrow[dl] &  & G(\phi_k,\{\psi_{k-1}(i),\psi_{k-1}(j)\})\arrow[dr]\arrow[dl] & \\
 G ([n_0], \{i,j\}) & & G ([n_1],\{\psi_1(i),\psi_1(j)\})  & \cdots & G([n_{k-1}],\{\psi_{k-1}(i),\psi_{k-1}(j)\}) & & G ([n_k],\{\psi_{k}(i),\psi_{k}(j)\})  
\end{tikzcd}
\end{equation}
where $\psi_i:=\phi_i\circ \phi_{i-1} \circ \cdots \circ \phi_1$, is a limit diagram in $ \CC$. 
\end{prop}

\begin{proof}
By definition, $G$ defines a functor 
\[
\overline{G}: \Delta\to \Span_\Delta (\CC^{\boxtimes})
\]
if and only if every restriction of $\tilde{G}$ to $\Tw(\Delta^n)\subset \Tw(\Delta)$ is a Segal simplex in $\CC^{\boxtimes}$. 

Let $\Delta^k\hookrightarrow \Delta$ be the simplex 
\[
[n_0] \overset{\phi_1}{\to}[n_1]\overset{\phi_2}{\to}\cdots \overset{\phi_k}{\to} [n_k].
\]
Then by \cite[Lemma 10.2.13]{DKHSSI}, there is a functor 
\[
H: \left( \Delta^1\times \Tw(\Delta^k)\right) \times_\Delta \Delta^{\amalg}\to  \CC
\]
representing a homotopy 
\[
\tilde{H}: \Delta^1\times \Tw(\Delta^k)\to \CC^{\boxtimes}.
\]
This homotopy has components that are Cartesian morphisms, and the component $\tilde{G}_0:=\tilde{H}|_{\{0\}\times \Tw(\Delta^k)}$ has image contained in $\CC^{\boxtimes}_{[n_0]}$. Since this is the case, the condition that $\tilde{G}$ is a $p$-limit diagram when restricted to the Segal cone is equivalent to the condition that $\tilde{G}_0$ is a limit diagram in $\CC^{\boxtimes}_{[n_0]}$ when restricted to the Segal cone. This can be checked componentwise, using one component for each subinterval of $[n_0]$. 

Fix one such subinterval, $\{i,j\}$. Then the coresponding Segal cone diagram in $ \CC$ will be 
\[
\begin{tikzcd}[column sep= -20pt]
 & & & G_0(\phi_n\circ \cdots \circ \phi_1,\{i,j\})\arrow[dll]\arrow[drr] & & & \\
 & G_0(\phi_1,\{i,j\})\arrow[dr]\arrow[dl] & &  \cdots\arrow[dr]\arrow[dl] &  & G_0(\phi_k,\{i,j\})\arrow[dr]\arrow[dl] & \\
 G_0 ([n_0], \{i,j\}) & & G_0 ([n_1],\{i,j\})  & \cdots & G_0([n_{k-1}],\{i,j\}) & & G_0 ([n_k],\{i,j\})  
\end{tikzcd}
\]
Since the homotopy has Cartesian components, $H$ will restrict to a natural equivalence between this diagram and the diagram $(1)$. Therefore, a simplex is Segal if and only if all such diagrams are limit diagrams. 
\end{proof}

\subsubsection{Cartesian morphisms and equivalences}
Suppose $G$ represents a coalgebra object. Given an inert morphism $\Delta^1\overset{\{\phi\}}{\to} \Delta$ ($\phi:{[n]}\to {[m]}$), $G$ must send $\phi$ to a Cartesian morphism in $\Span_\Delta(\CC^{\boxtimes})$. This means that the adjoint map 
\[
\Tw(\Delta^1)\to \CC^{\boxtimes} 
\]
is comprised only of Cartesian morphisms. Therefore:
\begin{itemize}
\item For the source map $\phi\to {[n]}$ in $\Tw(\Delta)$, and for any $\{i,j\}\in {[n]}$, the induced morphism 
\[
G(\phi,\{i,j\})\to G({[n]},\{i,j\})
\]
is an equivalence.
\item For the target map $\phi\to {[m]}$ in $\Tw(\Delta)$, and for any $\{i,j\}\in {[n]}$
The induced morphism 
\[
G(\phi,\{i,j\})\to G({[m]}, \{\phi(i),\phi(j)\})
\]
is an equivalence. 
\end{itemize}

We will write $\phi_{i,j}:[i,\ldots,j]\to {[n]}$ for the inert morphism which includes the interval $[i,\ldots, j]$. 

\begin{prop}\label{prop:pullbacksofequivs}
Suppose $G$ represents a coalgebra object. Let $f:{[n]}\to{[m]}$ be a morphism in $\Delta$, viewed as an object in $\Tw(\Delta)$. 
\begin{enumerate}
\item Let $f|_{\{i,j\}}:[i,\ldots, j]\to {[m]}$ be the restriction of $f$ to $[i,\ldots, j]\subset {[n]}$. Then the induced morphism 
\[
G(f|_{\{i,j\}},\{i,j\}) \to G(f,\{i,j\})
\]
is an equivalence. 
\item Let $\tilde{f}:{[n]}\to [i,\ldots, j] \subset {[m]}$ be a morphism such that composing with the inert morphism $\phi_{i,j}:[i,\ldots,j]\to {[m]}$ yields $f$. Then the induced morphism 
\[
G(f,\{i,j\}) \to G(\tilde{f},\{i,j\})
\]
is an equivalence.  
\end{enumerate} 
\end{prop} 

\begin{proof}
Applying our conclusion from above, we find that in case (1), the diagram
\[
\begin{tikzcd}
  & {G(f|_{\{i,j\}},\{i,j\})}\arrow[dl]\arrow[dr] &  \\
 G(\phi_{i,j},\{i,j\})\arrow[dr] & & G(f,\{i,j\})\arrow[dl]  \\
 & G(id_{{[n]}},\{i,j\}) &
\end{tikzcd}
\]
Must be pullback. Therefore, since  $G(\phi_{i,j},\{i,j\})\to G(\id_{{[n]}},\{i,j\})$ must be an equivalence, so must $G(f|_{\{i,j\}},\{i,j\}) \to G(f,\{i,j\})$.

Similarly, in case (2), the diagram
\[
\begin{tikzcd}
  & G(f,\{i,j\})\arrow[dl]\arrow[dr] &  \\
 G(\tilde{f},\{i,j\})\arrow[dr] & & G(\phi_{i,j},\{i,j\})\arrow[dl]  \\
 & G(id_{[i,\ldots,j]},\{i,j\})
\end{tikzcd}
\]
must be pullback. Therefore, since $G(\phi_{i,j},\{i,j\})\to G(\id_{[i,\ldots,j]},\{i,j\})$ must be an equivalence, so must $G(f,\{i,j\}) \to G(\tilde{f},\{i,j\})$. %% cf. "Dualizing cartesian and cocartesian fibrations" section 3.1 and 3.1. For S an infinity category, (S,S^\cong, S)  is an adequate triple of infinity categories
\end{proof}

\begin{lem}\label{lem:classE}
Suppose $G$ sends the morphisms from \cref{prop:pullbacksofequivs} to equivalences. Let
\[
\mu:=\left\lbrace \vcenter{\vbox{\begin{tikzcd}
{{{[k]}}}\cong\{i,j\}\arrow[Subseteq]{r}{} & {[n]}\arrow[r,"f"]\arrow[d,"g"] & {[m]} \\
{{[k]}}\cong\{i^\prime,j^\prime\}\arrow[Subseteq]{r}{} & {[n^\prime]}\arrow[r,"f^\prime"] & {[m^\prime]}\arrow[u,"\overline{g}"]
\end{tikzcd}}}
\mkern-300mu\right\rbrace
\]
be a morphism such that $g$ restricts to an isomorphism $[i,\ldots,j]\overset{\cong}{\to} [i^\prime,\ldots, j^\prime]$ and $\overline{g}$ restricts to an isomorphism $[f^\prime(i^\prime), f^\prime(i^\prime)+1,\ldots, f^\prime(j^\prime)]\overset{\cong}{\to} [f(i), f(i)+1,\ldots, f(j)]$. Then $G$ sends $\mu$ to an equivalence.
\end{lem}

\begin{proof}
We first note that, under the given hypotheses, $G$ will send morphisms of the form 
\[
\nu:=\left\lbrace \vcenter{\vbox{\begin{tikzcd}
\{0,k\}\arrow[Subseteq]{r}{} & {{[k]}}\arrow[r,"s"]\arrow[d,"\id_{{{[k]}}}"'] & {[m]} \\
 \{0,k\}\arrow[Subseteq]{r}{} &{{[k]}}\arrow[r,"s^\prime"'] & {[m^\prime]}\arrow[u,"h"']
\end{tikzcd} }}
\mkern-350mu\right\rbrace
\]
to equivalences, where $h$ sends $[s^\prime(i^\prime), s^\prime(i^\prime)+1,\ldots, s^\prime(j^\prime)]$ isomorphically to ${[m]}$. This follows from composing 
\[
\begin{tikzcd}
\{0,k\}\arrow[Subseteq]{r}{} & {{[k]}}\arrow[r, "s"]\arrow[d,"\id_{{{[k]}}}"'] & {[m]} \\
 \{0,k\}\arrow[Subseteq]{r}{} &{{[k]}}\arrow[r,"s^\prime"]\arrow[d,"\id_{{{[k]}}}"'] & {[m^\prime]}\arrow[u,"h"']\\
 \{0,k\}\arrow[Subseteq]{r}{} & {{[k]}}\arrow[r,"s"'] & {[m]}\arrow[u,"\psi"'] 
\end{tikzcd}
\]
Where $\psi$ is the inclusion of the interval $[f(0),\ldots, f(k)]$. The lower morphism is then one of the morphisms of type (2) from \cref{prop:pullbacksofequivs} and the two morphisms compose to the identity. So, by 2-out-of-3, $\nu$ must be sent to an equivalence.

Now write $[\ell]:=[f(i), f(i)+1,\ldots, f(j)]$, and consider the composition 
\[
\begin{tikzcd}
\{0,k\}\arrow[Subseteq]{r}{} & {{[k]}}\arrow[r,"s"]\arrow[d,"\phi_{i,j}"] & {[\ell]} \\
\{i,j\}\arrow[Subseteq]{r}{} & {[n]}\arrow[r,"f"]\arrow[d,"g"] & {[m]}\arrow[u,"h"] \\
{{[k]}}\cong\{i^\prime,j^\prime\}\arrow[Subseteq]{r}{} & {[n^\prime]}\arrow[r,"f^\prime"] & {[m^\prime]}\arrow[u,"\overline{g}"]
\end{tikzcd} 
\]
Where $h$ sends $[\ell]$ isomorphically to itself. The upper morphism is the composite of a morphism of type (1) from \cref{prop:pullbacksofequivs} and a morphism of the same kind as $\nu$. Moreover, the composite
\[
\begin{tikzcd}
\{0,k\}\arrow[Subseteq]{r}{} & {{[k]}}\arrow[r,"s"]\arrow[d,"\phi_{i^\prime,j^\prime}"'] & {[\ell]} \\
{[k]}\cong\{i^\prime,j^\prime\}\arrow[Subseteq]{r}{} & {[n^\prime]}\arrow[r,"f^\prime"'] & {[m^\prime]}\arrow[u,"h^\prime"']
\end{tikzcd}
\]
is also the composite of a morphism of type (1) from  \cref{prop:pullbacksofequivs} and a morphism of the same kind as $\nu$. Therefore, by the 2-out-of-3 property, $\mu$ must be sent to an equivalence. 
\end{proof}

\begin{defn}
We define $E$ to be the set of all morphisms of the form from \cref{lem:classE}. Note that $E$ is closed under composition. 
\end{defn}

\begin{cor}\label{cor:algconds}
A functor $G:\Theta\to  \CC$ defines a coalgebra object if and only if 
\begin{enumerate}
\item\label{algconds:degint} $G$ sends degenerate intervals to the terminal object.
\item\label{algconds:product} $G$ sends $(\{i,j\}\subset {[n]}\overset{f}{\to}{[m]})$  together with its projections to sub-intervals to a product diagram.
\item\label{algconds:ELoc} $G$ sends the morphisms in $E$ to equivalences.
\item\label{algconds:Span} $G$ sends all diagrams of the form \cref{diag:segalsimpsinHSpan} to limit diagrams. 
\end{enumerate}
\end{cor}

\subsubsection{Forgetting degenerate intervals}

\begin{defn}
We denote by $\SAlg(\CC)$ the full sub-$\infty$-category of $\Fun (\Theta,  \CC)$ on those functors satisfying conditions (\ref{algconds:degint})-(\ref{algconds:Span}) from the corollary. We denote by $\Fun^\ast (\Theta,  \CC)$ the full sub-$\infty$-category of functors sending every degenerate interval to a terminal object in $ \CC$ (i.e., those functors satisfying condition (\ref{algconds:degint}) from the corollary).

Let $\Omega $ be the full subcategory of $\Theta$ on those objects $\{i,j\}\subset {[n]}\overset{f}{\to}{[m]}$ such that the interval $\{i,j\}$ is not degenerate (i.e. $i\neq j$). Pulling back along the inclusion $\Omega \to\Theta$ induces a functor $S: \Fun^\ast (\Theta,  \CC) \to \Fun (\Omega , \CC)$. 
\end{defn}

\begin{defn}
Given a 1-category $D$, call an object $d\in D$ \emph{attracting} if, for all $a\in D$, 
\[
\Hom_{D} (a,d)\neq \emptyset,\quad \text{and}\quad \Hom_{D} (d,a)=\emptyset.
\]
\end{defn}

\begin{lem}\label{lem:ForgetAttracting}
Let $d\in D$ be an attracting object, denote by $\Fun^\ast(D, \CC)$ the full sub-$\infty$-category on those functors sending $d$ to the terminal object, and denote by $D^\circ$ the full subcategory on all objects other than $d$. Then the functor 
\[
\Fun^\ast (D, \CC) \to \Fun (D^\circ,  \CC)
\]  
is an equivalence.
\end{lem}

\begin{proof}
Without loss of generality, we assume that $ \CC$ has a unique terminal object. 
when $f$ sends  $d$ to the terminal object. Denote by $ \CC^\prime\subset  \CC$ the largest subcategory not containing morphisms from the terminal object to any other object, and denote by $ \CC^\circ$ the full subcategory on non-terminal objects. Then we have an equivalence $ \CC^\prime\simeq ( \CC^\circ)^{\!\rotatebox{30}{$\mathsmaller{\triangle}$}}$ since the hom-spaces to the terminal object are all contractible. Any simplex in $\Fun^\ast (D, \CC)$ factors through $\Fun^\ast (D, \CC^\prime)$, so it will suffice to show that 
\[
\Fun^\ast(D,( \CC^\circ)^{\!\rotatebox{30}{$\mathsmaller{\triangle}$}}) \to (D^\circ,  \CC) 
\]
is a trivial Kan fibration. 

Unwinding the definitions, this amounts to solving the extension problem 
\[
\begin{tikzcd}
\left( \partial \Delta^n\times D\right) \coprod_{\partial \Delta^n\times D^\circ} \Delta^n \times D^\circ\arrow[dd]\arrow[dr,"f"] & \\
 & ( \CC^\circ)^{\!\rotatebox{30}{$\mathsmaller{\triangle}$}}\\  
 \Delta^n\times D\arrow[ur,dashed] &  
 \end{tikzcd}
\]
where $f$ sends $\partial \Delta^n\times D$ to the cone point. However, this implies that $f$ factors through $(\Delta^n\times D^\circ)^{\!\rotatebox{30}{$\mathsmaller{\triangle}$}}$. Pulling back along $\Delta^n\times D\to(\Delta^n\times D^\circ)^{\!\rotatebox{30}{$\mathsmaller{\triangle}$}}$ then gives the desired extension.  

\end{proof}

\begin{cor}
The functor $S: \Fun^\ast (\Theta,  \CC) \to \Fun (\Omega , \CC)$ is an equivalence of $\infty$-categories.
\end{cor}

\begin{proof}
We again assume that $ \CC$ has a unique terminal object. Let $\Theta^{deg}$ be the full subcategory on only the degenerate intervals. We can write $\Fun^\ast(\Theta,  \CC)$ as a pullback in $\Set_\Delta$
\[
\begin{tikzcd}
\Fun^\ast(\Theta,  \CC)\arrow[r]\arrow[d] & \Fun(\Theta,  \CC)\arrow[d]\\
\Fun(\Theta^{deg},\ast)\arrow[r] & \Fun(\Theta^{deg}, \CC)
\end{tikzcd}
\]
There is a natural transformation of diagrams to the pullback diagram
\[
\begin{tikzcd}
\Fun^\ast(\Theta\coprod_{\Theta^{deg}}\ast,  \CC)\arrow[r]\arrow[d] & \Fun(\Theta,  \CC)\arrow[d]\\
\Fun^\ast(\ast, \CC)\arrow[r] & \Fun(\Theta^{deg}, \CC)
\end{tikzcd}
\]
Since this natural transformation is an isomorphism on the bottom three objects, the universal property of the pullback gives us an isomorphism $\Fun^\ast(\Theta, \CC)\cong \Fun^\ast (\Theta\coprod_{\Theta^{deg}}\ast,  \CC)$. $\ast\in \Theta\coprod_{\Theta^{deg}}\ast$ is an attracting object, and so \cref{lem:ForgetAttracting} yields the desired result. 
\end{proof}
\subsection{The localization}

\begin{const}
Let $\phi: ([n],\{i,j\})\to ([m],\{k,\ell\})$ be a morphism in $\Delta^\amalg$. and write $\{i\leq j\}$ for the linearly ordered set $\{i,i+1,\ldots, j\}$. Applying $\mathbb{O}$ to $\phi$, we obtain a diagram  
\[
\begin{tikzcd}
\mathbb{O}([m])\arrow[r,"\mathbb{O}(\phi)"] & \mathbb{O}([n])\\
\mathbb{O}(\{k\leq\ell\})\arrow[Subseteq]{u}{} & \mathbb{O}(\{i\leq j\})\arrow[Subseteq]{u}{}\\
\mathbb{I}(\{k\leq\ell\})\arrow[Subseteq]{u}{} & \mathbb{I}(\{i\leq j\})\arrow[Subseteq]{u}{}
\end{tikzcd}
\]
Since, $\phi(i)\leq k\leq \ell\leq \phi(j)$, we see that, for every $a\in \{k\leq \ell\}$, there exists a $b\in \{i\leq j\}$ such that $\phi(b)\leq a\leq a+1\leq \phi(b+1)$. That is, $\mathbb{O}(\phi)$ descends uniquely to a map 
\[
\res(\phi):\mathbb{I}(\{k\leq\ell\})\to \mathbb{I}(\{i\leq j\}).
\]
Note that we here apply the convention that $\mathbb{I}([0])=\emptyset$. We therefore obtain a functor 
\[
\res: \Delta^\amalg \to \Delta_+^\op
\]
which sends all non-degenerate intervals into $\Delta\subset \Delta_+$. 
\end{const}

\begin{defn}
Define a category $\Delta^\star$ to have objects finite (non-empty) ordered tuples of elements in $\Delta$. The morphisms of $\Delta^\star$ from $([n_0],\ldots,[n_k])\to ([m_0],\ldots, [m_\ell])$ consist of 
\begin{enumerate}
\item A morphism $\phi:[\ell] \to [k]$ in $\Delta$. 
\item For each $i\in \{0,1,\ldots k\}$, with $\phi^{-1}(i)=(j_1,\ldots, j_r)$, a morphism 
\[
f_i: [m_{j_1}]\star [m_{j_2}]\star\cdots \star [m_{j_r}]\to [n_i]
\]
in $\Delta$.
\end{enumerate}
Satisfying the conditions that
\begin{enumerate}
\item If there is a $p\in \langle \ell\rangle^\circ$ with $r>\max_{j\in\phi^{-1}(i)}(j)$, then $f_i$ hits $n_i\in [n_i]$.
\item If there is a $p\in \langle \ell\rangle^\circ$ with $r<\min_{j\in\phi^{-1}(i)}(j)$, then $f_i$ hits $0\in [n_i]$.
\end{enumerate}
\end{defn}

\begin{rmk}
We could equivalently define the morphisms to be 
\begin{enumerate}
\item A morphism $\phi:[\ell] \to [k]$ in $\Delta$. 
\item A morphism 
\[
f: [m_1]\star [m_2]\star\cdots \star [m_\ell]\to [n_1]\star [n_2]\star \cdots \star [n_k]
\]
in $\Delta$.
\end{enumerate}
Satisfying the condition that, for any $i\in [k]$ with $\phi^{-1}(i)=(j_1,\ldots, j_r)$, the restriction 
\[
f_i: [m_{j_1}]\star [m_{j_2}]\star\cdots \star [m_{j_r}]\to [n_1]\star [n_2]\star \cdots \star [n_k]
\]
has image contained in $[n_i]$. 
\end{rmk}

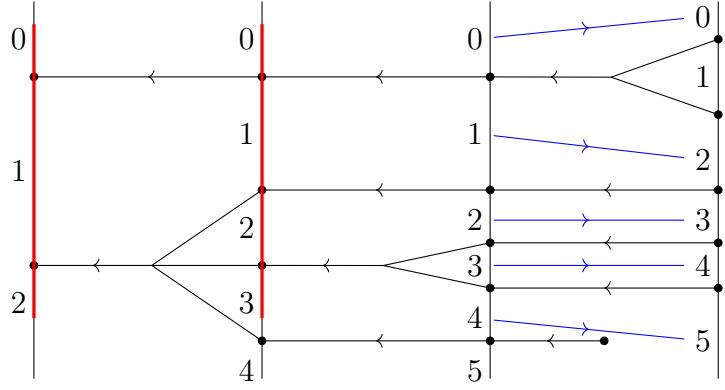
\begin{figure}
	\begin{tikzpicture}[decoration={
		markings,
		mark=at position 0.5 with {\arrow{>}}}
	]
	\foreach \x in {0, 3, 6, 9}{
		\draw (\x,-0.5) -- (\x, 4.5);
	};
	\path (0,1) coordinate (l2);
	\path (0,3.5) coordinate (l1);
	
	\path (3,3.5) coordinate (lm1);
	\path (3,2) coordinate (lm2);
	\path (3,1) coordinate (lm3);
	\path (3,0) coordinate (lm4);
	
	\path (6,3.5) coordinate (rm1);
	\path (6,2) coordinate (rm2);
	\path (6,1.3) coordinate (rm3);
	\path (6,0.7) coordinate (rm4);
	\path (6,0) coordinate (rm5);
	
	\path (9,4) coordinate (r11);
	\path (9,3) coordinate (r12);
	\path (9,2) coordinate (r2);
	\path (9,1.3) coordinate (r3);
	\path (9,0.7) coordinate (r4);
	
	\foreach \lab in {l1,l2,lm1,lm2,lm3,lm4,rm1,rm2,rm3,rm4,rm5,r11,r12,r2,r3,r4}{
		\draw[fill=black] (\lab) circle (0.05);
	};
	\draw[postaction={decorate}] (rm1) to (lm1); 
	\draw[postaction={decorate}] (lm1) to (l1);
	
	\draw[postaction={decorate}] (rm2) to (lm2);
	\draw[postaction={decorate}] (rm5) to (lm4);
	
	\draw[postaction={decorate}] (r2) to (rm2);
	\draw[postaction={decorate}] (r3) to (rm3);
	\draw[postaction={decorate}] (r4) to (rm4);
	\draw[postaction={decorate}] ($(1.5,0)+(rm5)$) to (rm5);
	\draw[fill=black] ($(1.5,0)+(rm5)$) circle (0.05);
	
	\coordinate (vert) at ($ (-0.7,0)+0.25*(lm2)+0.25*(lm3)+0.25*(lm4)+0.25*(l2)$);
	\draw[postaction={decorate}] (vert) to (l2); 
	\draw (lm2)--(vert) (lm3)--(vert) (lm4)--(vert); 
	
	\coordinate (vert) at ($ (-0.4,0)+0.333*(rm3)+0.333*(rm4)+0.333*(lm3)$);
	\draw[postaction={decorate}] (vert) to (lm3); 
	\draw (rm4)--(vert) (rm3)--(vert); 
	
	\coordinate (vert) at ($ (-0.4,0)+0.333*(r11)+0.333*(r12)+0.333*(rm1)$);
	\draw[postaction={decorate}] (vert) to (rm1); 
	\draw (r11)--(vert) (r12)--(vert);
	
	%%% interstices
	\path (-0.2,4) node (rint1) {0};
	\path (-0.2,2.25) node (rint2) {1};
	\path (-0.2,0.5) node (rint3) {2};
	\draw[very thick, red] ($(rint1)+(0.2,0.2)$) to ($(rint3)+(0.2,-0.2)$);
	
	\path (2.8,4) node (rmint1) {0};
	\path (2.8,2.75) node (rmint2) {1};
	\path (2.8,1.5) node (rmint3) {2};
	\path (2.8,0.5) node (rmint4) {3};
	\path (2.8,-0.4) node (rmint5) {4};
	\draw[very thick, red] ($(rmint1)+(0.2,0.2)$) to ($(rmint4)+(0.2,-0.2)$);
	
	\path (5.8,4) node (lmint1) {0};
	\path (5.8,2.75) node (lmint2) {1};
	\path (5.8,1.6) node (lmint3) {2};
	\path (5.8,1) node (lmint4) {3};
	\path (5.8,0.3) node (lmint5) {4};
	\path (5.8,-0.4) node (lmint6) {5};
	
	\path (8.8,4.3) node (lint1) {0};
	\path (8.8,3.5) node (lint2) {1};
	\path (8.8,2.4) node (lint3) {2};
	\path (8.8,1.6) node (lint4) {3};
	\path (8.8,1) node (lint5) {4};
	\path (8.8,0) node (lint6) {5};
	
	\begin{scope}[blue]
	\draw[postaction={decorate}] (lmint1) to (lint1);
	\draw[postaction={decorate}] (lmint2) to (lint3);
	\draw[postaction={decorate}] (lmint3) to (lint4);
	\draw[postaction={decorate}] (lmint4) to (lint5);
	\draw[postaction={decorate}] (lmint5) to (lint6);
	\end{scope}
	\end{tikzpicture}
	\caption{A pictorial representation of a morphism $\mu$ in $\Omega$, viewed as a triple of composable morphisms $[n]\overset{g}{\to}[n^\prime] \overset{f^\prime}{\to} [m^\prime] \overset{\overline{g}}{\to} [m]$ in $\Delta$. The dual forest is drawn in black, the chosen subintervals of $[n]$ and $[n^\prime]$ marked in red, and the induced morphism $\mathcal{L}(\mu)$ is drawn in blue. Note that that source of $\mathcal{L}(\mu)$ is the imbrication of the sets $\{f^\prime (i),f^\prime (i)+1,\ldots f^\prime (i+1)\}$. } \label{fig:DeltaLocMorphs}
\end{figure}

\begin{const}
We now define a functor $\mathcal{L}:\Omega\to \Delta^\star$. On objects it is given by 
\[
\{i,j\}\subset[n]\overset{f}{\to} [m]\mapsto \left( \{f(i)\leq f(i+1)\},\ldots, \{f(j-1)\leq f(j)\}\right)  
\]
where  $\{f(k)\leq f(k+1)\}:=\{f(k),f(k)+1),\ldots, f(k+1)\}$ are considered to be ordered via the order on $[m]$. Note that the indexing set of $\mathcal{L}(\{i,j\},f)$ is precisely $\mathbb{I}(\{i\leq j\})$

On morphisms, $\mathcal{L}$ is more complicated. A morphism in $\Omega$ is given by a commutative diagram of the form 
\[
\mu= \left\lbrace \vcenter{\vbox{\xymatrix{
[k]:=\{i,j\}\ar@{}[r]|-*[@]{\subset} & [n]\ar[r]^{f}\ar[d]_{g} & [m] \\
[k^\prime]:=\{i^\prime,j^\prime\}\ar@{}[r]|-*[@]{\subset} & [n^\prime]\ar[r]_{f^\prime} & [m^\prime]\ar[u]_{\overline{g}}
} }}\right\rbrace 
\]
where $g(i)\leq i^\prime \leq j^\prime \leq g(j)$. We define $\mathcal{L}(\mu)$ to be a pair $(\phi_f,\psi_f)$. We then write $\phi_f:=\res (g):\mathbb{I}(\{i^\prime\leq j^\prime\})\to \mathbb{I}(\{i\leq j\})$.  

Since the diagram commutes, for each pair $\{p,p+1\}\subset\{i,j\}\subset [n]$, we have that $\overline{g} (f(g(p)))= p$ and $\overline{g} (f(g(p+1)))= p+1$, so that $\overline{g}$ descends to a map of ordered sets
\[
\overline{g}_p:\{f^\prime (g(p)\leq f^\prime (g(p)+1)\}\star \cdots \star \{ f^\prime (g(p+1)-1)\leq f^\prime (g(p+1))\}\to \{f(p) \leq f(p+1)\}
\]
It is easy to verify that conditions (1) and (2) from the definition of $\Delta^\star$ are satisfied by the $\overline{g}_p$. On morphisms, therefore, we define 
\[
\mathcal{L} (\mu):= \left(\phi(g), \left\lbrace\overline{g}_p \right\rbrace_{i\leq p<j} \right).
\]
This is functorial via the functoriality of $\res$ and the restriction of $\overline{g}$. 
\end{const}

\subsubsection{Decomposing morphisms}
\begin{const}
Given a morphism 
\[
f:[m]\to[n]
\]
in $\Delta$, we can uniquely decompose it as follows: Let $[1]=:[1_i]\subset [m]$ be the interval $\{i-1\leq i\}$, and let $[n_i]\subset [n]$ be the interval $\{f(i-1)\leq f(i)\}$. Moreover, let $[n_{left}]$ and $[n_{right}]$ be the intervals $\{0\leq f(0)\}$ and $\{f(m)\leq n\}$ in $[n]$ respectively. Then $f$  is completely determined by the decomposition of $[n]$, since, given such a decomposition, we can reconstruct $f$ by defining $f_i:[1_i]\to [n_i]$ to be the unique map preserving maximal and minimal elements, so that $f$ is the composition 
\[
f=f_1\star\cdots \star f_m:[1_1]\star\cdots \star[1_m] \to [n_1]\star\cdots [n_m]\hookrightarrow [n_{left}]\star [n_1]\star\cdots [n_m]\star[n_{right}].
\]
We can clarify the indexing of the decomposition of $[n]$ by noting that the pairs $(i-1,i)$ considered above are precisely the inner interstices of $[m]$. Hence, we have decomposed $f$ as a morphism 
\[
\bigstar_{(i-1,i)\in \mathbb{I}([m])} \{i-1,i\} \to \bigstar_{(i-1,i)\in \mathbb{I}([m])} [n_i].
\] 
\end{const}

\begin{defn}
Given a morphism $\gamma:[n] \to [m]$ in $\Delta$, we can uniquely factor $\gamma$ as 
\[
[n]\overset{\gamma_1}{\to} [m_\gamma]\overset{\gamma_2}{\hookrightarrow} [m]
\]
where $[m]=[k]\oplus[m_\gamma]\oplus[\ell].$ Applying $O$, we get 
\[
O([m])\to O([m_\gamma])\to O)[n].
\] 
Where $O([m])\to O([m_\gamma])$ acts as projection onto a sub-interval. We call $O([m_\gamma])$ the \emph{minimal interval} of $\gamma$.
\end{defn}

\begin{lem}\label{lem:endpointambiguity}
Given an interval $\{i,j\}\subset [n]$ and a morphism $\eta:(\{i,j\}\subset[n])\to (\{r,r+k\} \subset [m])$  in $\Delta^{\amalg}$, let $[p,\ldots, q]$ be the minimal interval of $\gamma:=\res(\eta)$. Then $\eta|_{[p+1,\ldots,q-1]}= O(\gamma)|_{[p+1,\ldots,q-1]}$. 
\end{lem}

\begin{proof}
If $[p+1,\ldots,q-1]$ is empty, the statement is vacuously true. Otherwise, note that for $s\in[p+1,\ldots,q-1]$, the requirement that $\res(\eta)=\gamma$  means that $\gamma(\eta(s))\leq s<\gamma(\eta(s)+1)$. Such an $\eta(s)$ always exists, and this inequality uniquely determines $\eta(s)$. (Note that, for $p$ or $q$ in $[p,\ldots,q]$, we only have one-half of the inequality so that uniqueness need not hold.)  
\end{proof}

\subsubsection{Constructing morphisms}

In what follows, we will be interested in the weak fibers of the functor $\mathcal{L}:\Omega \to\Delta^\star$. We first note that, given an object $M=([m_1],\ldots,[m_k])\in\Delta^\star$, the fiber $\Omega_M$ is non-empty. We can explicitly build an object
\[
 \{0,k\} \subset [k]\overset{f_M}{\to} [m_1]\star [m_2]\star\cdots \star[m_k]=:[m]
\] 
in the fiber over $M$, given by 
\[
f_M(i)= \begin{cases}
0\in[m_{i+1}] & i<k\\ 
m_k \in [m_k] & i=k.
\end{cases}
\]

\begin{defn}
For  $M=([m_1],\ldots,[m_k])\in\Delta^\star$, we define a subcategory $\Omega_M^E\subset \Omega_M $ as follows. The objects of $\Omega_M^E$ are the same as those of $\Omega_M$, but the morphisms are only those in $E$. 
\end{defn}

\begin{lem}\label{lem:contractibleEfibers}
The object $\{0,k\}\subset [k]\overset{f_M}{\to} [m]$ is an initial object in $\Omega_M^{E}$. 
\end{lem}

\begin{proof}
Given another object
\[
\{i,j\}\subset [n]\overset{f}{\to} [m^\prime] 
\]
in $\Omega_M^E$, and a morphism
\[
\xymatrix{
\{0,k\}\ar@{}[r]|-*[@]{\subset} & [k]\ar[r]^{f_M}\ar[d]_{\phi} & [m] \\
\{i,j\}\ar@{}[r]|-*[@]{\subset} & [n]\ar[r]_{f} & [m^\prime]\ar[u]_{h}
}
\]
$\phi$ must be the inclusion of $[i,\ldots, j]$, since any such morphism in $E$ will induce an isomorphism $[k]\to [i,\ldots, j]$. Moreover, $h$ is clearly uniquely determined by the condition that it maps $[f(i),f(i+1), \ldots, f(j)]$ isomorphically to $[m]$. 
\end{proof}

Suppose given an object 
\[
Z:=\left\lbrace \{i,j\}\subset[n]\overset{f}{\to} [\ell]\right\rbrace
\]
in $\Omega $ whose image under $\mathcal{L}$ is $([\ell_{i+1}],\ldots,[\ell_j])$, and a morphism 
\[
g: ([\ell_{i+1}],\ldots,[\ell_j])\to ([m_0],\ldots, [m_{k-1}])
\]
in $\Delta^\star$. Write $\gamma:[k-1]\to [i+1,\ldots, j] \in \Delta$ and $\overline{g}:[m_0]\star\cdots\star[m_{k-1}]\to[\ell_{i+1}]\star\cdots\star[\ell_j]$ for the morphisms defining $g$. Denote by $[n_c]:=[p,\ldots,q]\subset\{i,j\}\subset [n]$ the minimal interval of $\gamma$ and by $\psi: [i,\ldots,j]\to[n_c]$ the projection as above, and let $\{0,k\}\subset [k]\overset{f_M}{\to} [m]:=[m_0]\star\cdots\star[m_{k-1}]$ be the minimal object in $\Omega $ representing the target. 

Note that, by definition, the morphism $\overline{g}$ has image contained in $[\ell_{p+1}]\star\cdots \star[\ell_q]=:[\ell_c]$.  We introduce some notation for specific decompositions:
\[
[n]  =  [n_\ell]\star[n_c]\star[n_r] 
\]
\[
[\ell]  =  [{\ell_\ell}]\star[{\ell_c}]\star[{\ell_r}]
\]

\begin{lem}\label{lem:cofinalinclusion}
There is a morphism in $\Omega $ 
\[
\xymatrix{
\{p,q\}\ar@{}[r]|-*[@]{\subset} & [p,\ldots,q]\ar[r]^{f|_{\{p,q\}}}\ar[d]_{\nu} & [\ell_c] \\
\{0,k\}\ar@{}[r]|-*[@]{\subset} & [1]\star[k]\star[1]\ar[r]_{f_M^\prime} & [\ell^1]\star[m]\star[\ell^2]\ar[u]_{\overline{g}^\prime}
}
\]
which extends to a morphism $\mu_{Z,M}$ in $\Omega $ covering $g$
\[
\mu_{Z,M}:=\left\lbrace\vcenter{\vbox{\xymatrix{
Z\ar@{}[r]|-*[@]{=}&\{i,j\}\ar@{}[r]|-*[@]{\subset} & [n_\ell]\star[n_c]\star[n_r] \ar[r]\ar[d] & [{\ell_\ell}]\star[{\ell_c}]\star[{\ell_r}] \\
Z_M\ar@{}[r]|-*[@]{:=}&\{0,k\}\ar@{}[r]|-*[@]{\subset} & [n_\ell]\star[1]\star[k]\star[1]\star[n_r] \ar[r] & [{\ell_\ell}]\star[\ell^1]\star[m]\star[\ell^2]\star[{\ell_r}]\ar[u]
}}}\right\rbrace
\]
Moreover, given any other morphism $Z\to X$ covering $g$, there is a unique morphism $Z_M\to X$ in $E$ such that the diagram 
\[
\xymatrix{ 
& Z\ar[dr]\ar[dl] & \\
Z_M\ar[rr] & & X \\
}
\]
commutes.
\end{lem}

\begin{proof}
In the first diagram, we define the map $\nu$ on $[p+1,\ldots, q-1]$ to be the unique map from \cref{lem:endpointambiguity} dual to $\gamma$ under $\res$, and send the endpoints to the endpoints of $[1]\star[k]\star[1]$. Then we write 
\[
[\ell_c]=[\ell^1]\star[\ell_c^m]\star[\ell^2],
\]
where $[\ell_c^m]$ is the minimal interval containing the image of $\overline{g}:[m]\to[\ell]$. Note that $\overline{g}:[m]\to [\ell_c^m]$ hits both endpoints. We then define 
\[
\overline{g}^\prime:=\id_{[\ell^1]}\star \overline{g}\star\id_{[\ell^2]}:[\ell^1]\star[m]\star[\ell^2]\to [\ell_c]
\]
(which then, by definition, hits both endpoints), and 
\[
f_M^\prime: [1]\star[k]\star[1]\to [\ell^1]\star[m]\star[\ell^2]
\]
to be $f_M$ on $[k]$, and to send endpoints to endpoints. Then we can decompose the diagram as 
\[
\xymatrix{
\{p,q\}\ar@{}[r]|-*[@]{\subset} & [1_{p+1}]\star\cdots\star[1_q]\ar[r]^{f|_{\{p,q\}}}\ar[d]_{\nu} & [\ell_{p+1}]\star\cdots \star[\ell_q] \\
\{0,k\}\ar@{}[r]|-*[@]{\subset} & [1]\star[k_{p+1}]\star\cdots\star[k_q]\star[1]\ar[r]_{f_M^\prime} & [\ell^1]\star[m_{p+1}]\star\cdots\star[m_q]\star[\ell^2]\ar[u]_{\overline{g}^\prime}
}
\]
by decomposing the morphisms $\nu$, $f|_{\{p,q\}}$, and $f_M^\prime\circ \nu$. The condition that the diagram commute is then equivalent to the conditions that, (1) for each $r\in\{p+2,\ldots, q-1\}$, the endpoints of $[m_r]$ are sent to the endpoints of $[\ell_r]$ by $\overline{g}$, and (2) that $\overline{g}$ sends the endpoints of $[\ell^1]\star[m_{p+1}]$ and $[m_q]\star[\ell^2]$ to the endpoints of $[\ell_{p+1}]$ and $[\ell_q]$, respectively. Since
\[
[m_r]=\bigstar_{a\in\mathbb{I}([k_r])} [f_M^\prime (a-1),f_M^\prime (a-1)+1,\ldots, f_M^\prime (a)]
\]
we see that case (1) is true by the definition of $\Delta^\star$.  Case (2) is true by construction.

This diagram is defined so that the maps $\nu$, $f_M^\prime$, $\overline{g}^\prime$, and $f|_{\{p,q\}}$ preserve endpoints. Therefore, we can take the appropriate star products with the morphisms $\id_{[n_\ell]}$, $\id_{[n_r]}$, $\id_{[\ell_\ell]}$, $\id_{[\ell_r]}$, $f|_{[n_\ell]}:[n_\ell]\to[\ell_\ell]$, and $f|_{[n_r]}:[n_r]\to[\ell_r]$ to get a commutative diagram
\[
\xymatrix{
Z\ar@{}[r]|-*[@]{=}&\{i,j\}\ar@{}[r]|-*[@]{\subset} & [n_\ell]\star[n_c]\star[n_r] \ar[r]\ar[d] & [{\ell_\ell}]\star[{\ell_c}]\star[{\ell_r}] \\
Z_M\ar@{}[r]|-*[@]{:=}&\{0,k\}\ar@{}[r]|-*[@]{\subset} & [n_\ell]\star[1]\star[k]\star[1]\star[n_r] \ar[r] & [{\ell_\ell}]\star[\ell^1]\star[m]\star[\ell^2]\star[{\ell_r}]\ar[u]
}
\]
By construction, the morphism $\res(\nu):[k-1] \to \langle i+1,\ldots, j\rangle$ is $\gamma$, and the morphism $\overline{g}^\prime$ restricts to $\overline{g}$ on $[m]$, so this diagram determines a morphism in $\Omega $ covering $g$. Call this morphism $\mu_{Z,M}:Z\to Z_M$.

Now suppose we are given a morphism 
\[
\xymatrix{
Z\ar@{}[r]|-*[@]{=}&\{i,j\}\ar@{}[r]|-*[@]{\subset} & [n]\ar[r]^{f}\ar[d]_{\rho} & [\ell] \\
X\ar@{}[r]|-*[@]{:=}&\{0,k\}\ar@{}[r]|-*[@]{\subset} & [a]\ar[r]_{h} & [b]\ar[u]_{w}
}
\]
covering $g$. We can decompose this into
\[
\xymatrix{
Z\ar@{}[r]|-*[@]{=}&\{i,j\}\ar@{}[r]|-*[@]{\subset} & [n_\ell]\star[n_c]\star[n_r] \ar[r]^{f}\ar[d]_{\rho} & [{\ell_\ell}]\star[{\ell_c}]\star[{\ell_r}] \\
Z_M\ar@{}[r]|-*[@]{:=}&\{0,k\}\ar@{}[r]|-*[@]{\subset} & [a_\ell]\star[a_c]\star[a_r] \ar[r]_{h} & [{b_\ell}]\star[b_c]\star[{b_r}]\ar[u]_{w}
}
\]
Where $\{0,k\}\subset [a_c]$. By  \cref{lem:endpointambiguity}, we know that $\rho$ is uniquely determined on all of $[n_c]$ except the endpoints. This allows us to further decompose the diagram 
\[
\xymatrix{
[n_c] \ar[r]^{f}\ar[d]_{\rho} & [{\ell_c}] \\
[a_c] \ar[r]_{h} & [b_c]\ar[u]_{w}
}
\]
as a diagram where the bottom map is a star product with $f_M$. 
\[
\xymatrix{
[n_c] \ar[r]^{f}\ar[d]_{\rho} & [{\ell_c}] \\
[a_c^1]\star[k]\star[a_c^2] \ar[r]_{h} & [b_c^1]\star[m]\star[b_c^2]\ar[u]_{w}
}
\]
If there is  morphism $Z_M\to X$ in $E$ commuting with the morphisms $Z\to X$ and $\mu_{Z,M}$, it must, in particular, restrict to a commutative diagram
\[
\xymatrix{
[n_c] \ar[r]\ar[d]\ar[ddr] & [{\ell_c}] &\\
[1]\star[k]\star[1]\ar[r]\ar[dr] & [\ell^1]\star[m]\star[\ell^2]\ar[u] &\\
 & [a_c^1]\star[k]\star[a_c^2] \ar[r]_{h} & [b_c^1]\star[m]\star[b_c^2]\ar[uul]\ar[ul]
}
\]
Moreover, since the morphism is in $E$, the bottom square must restrict to the commutative diagram
\[
\xymatrix{
[k]\ar[r]^{f_M}\ar[d]_{\id} & [m]\\
[k]\ar[r]^{f_M} & [m]\ar[u]^{\id}
}
\]
As a result, the component morphism $[1]\star[k]\star[1]\to[a_c^1]\star[k]\star[a_c^2]$ is uniquely determined by the commutativity of the left-hand triangle. Additionally, since $w:[b]\to [\ell]$ must restrict to $\overline{g}$ on $[m]$, we can decompose $w$ as a star product
\[
w=w^1 \star\overline{g}\star w^2: [b_c^1]\star[m]\star[b_c^2]\to [\ell^1]\star[\ell_c^m]\star[\ell^2]
\]
Therefore, the component morphism 
\[
[b_c^1]\star[m]\star[b_c^2]\to [\ell^1]\star[m]\star[\ell^2]
\]
is uniquely determined, and must be $w^1\star \id_{[m]}\star w^2$. 

We now extend back to the full diagram
\[
\xymatrix{
[n_\ell]\star[n_c]\star[n_r] \ar[r]\ar[d]\ar[ddr] &[\ell_\ell]\star [{\ell_c}]\star[\ell_r] &\\
[n_\ell]\star[1]\star[k]\star[1]\star[n_r]\ar[r]\ar[dr] & [\ell_\ell]\star[\ell^1]\star[m]\star[\ell^2]\star[\ell_r]\ar[u] &\\
 & [a_\ell]\star [a_c^1]\star[k]\star[a_c^2]\star[a_r]  & [b_\ell]\star[b_c^1]\star[m]\star[b_c^2]\star[b_r]\ar[uul]\ar[ul]
}
\]
and note that, since the vertical components of the back square restrict to identities on $[n_\ell]$, $[n_r]$, $[\ell_\ell]$, and $[\ell_r]$, the bottom square is uniquely determined by the morphisms $[n_{\ell}]\to[a_{\ell}]$, $[n_r]\to [a_r]$, $b_\ell]\to [\ell_\ell]$, and $[b_{r}]\to [\ell_{r}]$. So there is a unique morphism $Z_M\to X$ in $\Omega $ with the desired properties. 
\end{proof}

\begin{prop}
The functor $\mathcal{L}:\Omega \to\Delta^\star$ is an $\infty$-categorical localization at the morphisms in $E$.
\end{prop}

\begin{proof}
Consider the inclusion $\iota_M:\Omega_M^E\subset \Omega_M\hookrightarrow \Omega_{/M}$, and  
\[
Z:= \begin{cases}
 \{i,j\}\subset [n] \overset{f_Z}{\to} [\ell] & +\\
 g_Z: ([\ell_i],\ldots,[\ell_j]) \to ([m_1],\ldots,[m_k]) & \text{ in } \Delta^\star  
\end{cases}
\]
in $\Omega_{/M}$.  Denote the overcategory $\left( \Omega^{E}_{M}\right)_{Z/}:=\Omega^{E}_M\times_{\Omega_{/M}} \left( \Omega_{/M}\right)_{Z/}$.  \cref{lem:cofinalinclusion} tells us that $\left( \Omega^{E}_{M}\right)_{Z/}$ is non-empty, and that the object $Z_M$ constructed in the lemma is an initial object. Moreover, by  \cref{lem:contractibleEfibers}, $\Omega_M^E$ has an initial object. Therefore, by \cite[Lemma 3.1.1]{Walde}, $\mathcal{L}$ is a localization at the morphisms of $E$. 
\end{proof}

\subsubsection{Algebra conditions}
Denote by $\Fun^{alg} (\Delta^\star,\mathcal{C})$ the full sub-$\infty$-category of functors  $f$ which 
\begin{enumerate}[label=(\Alph*)]
\item \label{Deltaalgconds:span} send the diagrams 
\[
\xymatrix@!0@C=7em@R=5em{
 & (\bigstar_{j\in[n_1]}[m_j],\ldots, \bigstar_{j\in[n_\ell]}[m_j])\ar[dr]\ar[dl] & \\
([n_1],\ldots,[n_\ell]) \ar[dr] & & ([m_1],\ldots,[m_k])\ar[dl] \\
 & (\underbrace{[1],\ldots,[1]}_{\times k}) & 
}
\]
opposite the diagrams 
\[
\xymatrix@!0@C=7em@R=3em{
 & \bigstar [m_i] & \\
[n_1]\star\cdots\star[n_\ell]\ar[ur]  & & [m_1]\star\cdots\star[m_k] \ar[ul] \\
 & [1]\star\cdots\star[1] \ar[ul]^{\operatorname{id}} \ar[ur]_{\{0,m_1\},\ldots \{0,m_k\}} & 
}
\] 
to pullback diagrams. 
\item\label{Deltaalgconds:prod} send the diagrams 
\[
\xymatrix{
 & & ([m_1],\ldots,[m_k])\ar[drr]\ar[dr] \ar[dl]\ar[dll] & & \\
 [m_1] & [m_2] & \cdots & [m_{k-1}] & [m_k]
}
\]
to product diagrams. 
\end{enumerate} 

\begin{prop}\label{prop:DeltaAlgCondsinStar}
There is an equivalence of $\infty$-categories
\[
\SAlg(\mathcal{C})\simeq \Fun^{alg}(\Delta^\star,\mathcal{C}).
\]
\end{prop}

\begin{proof}
It is clear that condition \ref{Deltaalgconds:prod} corresponds to condition (\ref{algconds:product}) from \cref{cor:algconds}. For condition \ref{Deltaalgconds:span}, first consider a 3-simplex $[n_0]\overset{\phi_1}{\to} [n_1]\overset{\phi_2}{\to} [n_2]\overset{\phi_3}{\to}[n_3]$ in $\Delta$. The corresponding limit diagram \cref{diag:segalsimpsinHSpan} can be written as 
\begin{equation*}
\vcenter{\vbox{\xymatrix@!0@C=7em@R=3em{
 & & G(\phi_2\circ \phi_1,\{i,j\})\ar[dll]\ar[drr]\ar[d] & & \\
 G(\phi_1,\{i,j\})\ar[dr] & &  G(\phi_2, \{\phi_1(i),\phi_1(j)\}) \ar[dr]\ar[dl] &  & G(\phi_3,\{\psi_{2}(i),\psi_{2}(j)\})\ar[dl] \\
 & G ([n_1],\{\psi_1(i),\psi_1(j)\})  &  & G([n_{2}],\{\psi_{2}(i),\psi_{2}(j)\}) & 
}}}
\end{equation*}
However, by (the dual of) \cite[Proposition 4.4.2.2]{LurieHTT}, this diagram is a limit if and only if the induced diagram 
\begin{equation*}
\vcenter{\vbox{\xymatrix@!0@C=7em@R=3em{
 & & G(\phi_2\circ \phi_1,\{i,j\})\ar[dl]\ar[dr] & & \\
 & G ([n_1],\{\psi_1(i),\psi_1(j)\})\ar[dr]   &  & G([n_{2}],\{\psi_{2}(i),\psi_{2}(j)\})\ar[dl]  & \\
 & &  G(\phi_3, \{\phi_1(i),\phi_1(j)\})  &  &  
}}}
\end{equation*}
is pullback. However, combining these two diagrams, we get
\begin{equation*}
\vcenter{\vbox{\xymatrix@!0@C=7em@R=3em{
 & & G(\phi_2\circ \phi_1,\{i,j\})\ar[dl]\ar[dr] & & \\
 & G ([n_1],\{\psi_1(i),\psi_1(j)\})\ar[dl]\ar[dr]   &  & G([n_{2}],\{\psi_{2}(i),\psi_{2}(j)\})\ar[dl]\ar[dr]  & \\
 G(\phi_1,\{i,j\})\ar[dr] & &  G(\phi_3, \{\phi_1(i),\phi_1(j)\}) \ar[dr]\ar[dl] &  & G(\phi_k,\{\psi_{2}(i),\psi_{2}(j)\})\ar[dl] \\
 & G ([n_1],\{\psi_1(i),\psi_1(j)\})  &  & G([n_{2}],\{\psi_{2}(i),\psi_{2}(j)\}) & 
}}}
\end{equation*}
By the pasting property for pullback diagrams, we thus see that it is sufficient to require that each of the diagrams corresponding to the sub-2-simplices of our simplex is pullback. Iterating this argument, we find that property (\ref{algconds:Span}) of corollary \ref{cor:algconds} is satisfied if and only if it is satisfied on 2-simplices. Since condition \ref{Deltaalgconds:span} is the image of this 2-simplex condition under $\mathcal{L}$, this proves the proposition.
\end{proof}

\begin{lem}
A functor $f\in \Fun(\Delta^\star,\mathcal{C})$ satisfies condition \ref{Deltaalgconds:span} if and only if it satisfies condition \ref{Deltaalgconds:span} for collections where all but one of the $[m_i]$ are equal to $[1]$.
\end{lem}

\begin{proof}
This follows from applying the pasting law to diagrams of the form
\[
\xymatrix{
([n_1],\ldots,[n_\ell])\ar[d] & ([m_1]\star [1]\star\cdots\star[1], [1],\ldots,[1])\ar[d]\ar[l] & ([m_1]\star[m_2]\star\cdots \star[1],[1],\ldots, [1]) \ar[d]\ar[l] & \cdots\ar[l]\\
([1],\ldots, [1]) & ([m_1],[1],\ldots ,[1])\ar[d]\ar[l] & ([m_1], [m_2],\ldots, [1])\ar[d]\ar[l] & \cdots \ar[l]\\
 & ([1],\ldots, [1]) & ([1],\ldots,[1],[m_2], [1],\ldots, [1])\ar[l] & \cdots\ar[l]
}
\]
If condition \ref{Deltaalgconds:span} is satisfied for squares where all but one of the $[m_i]$ are equal to $[1]$, then the bottom right square and the right-hand rectangle are all pullback. Therefore, the top right square is pullback. Since our restricted version of condition \ref{Deltaalgconds:span} also implies that the top left square is pullback, the top rectangle is pullback. Iterating this argument then yields the lemma. 
\end{proof}

\subsection{Extension and restriction}

Considering the full subcategory of $\Delta^\star$ on the objects $([n])$ for $n\geq 0$ we get 
\[
\iota: \Delta^{op}\to \Delta^\star. 
\]
Taking restriction and right Kan extension gives us an adjunction of infinity categories
\[
\iota_\ast: \Fun(\Delta^\star,\CC) \leftrightarrow  \Fun(\Delta^{op},\CC): \iota_!
\]
Denote by $\Fun^\times(\Delta^\star,\CC)$ the full sub-$\infty$-category that sends each diagram 
\[
\xymatrix{
 & & ([m_1],\ldots,[m_k])\ar[drr]\ar[dr] \ar[dl]\ar[dll] & & \\
 [m_1] & [m_2] & \cdots & [m_{k-1}] & [m_k]
}
\]
to a limit diagram. 
\begin{prop}\label{prop:DeltaKE}
The adjunction $\iota_\ast: \Fun(\Delta^\star,\CC) \leftrightarrow  \Fun(\Delta^{op},\CC): \iota_!$ 
descends to an equivalence of $\infty$-categories 
\[
\Fun^\times(\Delta^\star,\CC) \simeq  \Fun(\Delta^{op},\CC).
\]
\end{prop}
\begin{proof}
We compute the overcategory $(\Delta^{op})_{([m_1],\ldots,[m_k])/}$. An object in the overcategory will consist of  a choice of $i\in \{1,2,\ldots ,k\}$ and a morphism $[n]\to [m_i]$. A morphism $(i,[n]\to [m_i])\to(j,[\ell]\to[m_j])$ only exists if $i=j$, and in this case is given by a commutative diagram 
\[
\xymatrix{ 
[\ell]\ar[rr]\ar[dr] & & \ar[dl][n]\\
 & [m_i] & 
}
\]
consequently, we find that the induced diagram 
\[
\xymatrix{
 & &(\Delta^{op})_{([m_1],\ldots,[m_k])/} & & \\
 (\Delta^{op})_{([m_1])/}\ar[urr] & (\Delta^{op})_{([m_2 ])/}\ar[ur]& \cdots & (\Delta^{op})_{([m_{k-1}])/}\ar[ul] &  (\Delta^{op})_{([m_k])/}\ar[ull]
}
\]
displays $(\Delta^{op})_{([m_1],\ldots,[m_k])/}$ as a coproduct, and, hence, for any $f\in \Fun(\Delta^{op},\CC)$, the diagram 
\begin{equation}\label{diag:liminKan}
\xymatrix{
 & & \iota_!f\left( ([m_1],\ldots,[m_k])\right)\ar[drr]\ar[dr] \ar[dl]\ar[dll] & & \\
\iota_!f\left(  [m_1]\right) & \iota_!f\left( [m_2]\right) & \cdots & \iota_!f\left( [m_{k-1}]\right) & \iota_!f\left( [m_k]\right)
}
\end{equation}
displays $\iota_!f\left( ([m_1],\ldots,[m_k])\right)$ as a product. Consequently, the adjunction descends to an adjunction $\iota_\ast: \Fun^\times(\Delta^\star,\CC) \leftrightarrow  \Fun(\Delta^{op},\CC): \iota_!$.

Since this is a right Kan extension from a full subcategory, the counit is an equivalence. Moreover, the components of the unit are equivalences on the objects of $\Delta^{op}$. However, for every object $([m_1],\ldots, [m_k])$, the unit induces a natural transformation of limit diagrams of the form in diagram (\ref{diag:liminKan}). Therefore, we see that the components of the unit are equivalences for all objects, and thus, the unit is also an equivalence.  
\end{proof}

%Let $\Fun^{\ast,\times,E}(\Omega,\CC)$ be the full subcategory of functors that send degenerate intervals to the terminal object, send \red{need a name here} diagrams to products, and localize at the morphisms in $E$. Combining our results, we get 
%\begin{cor} 
%There is an equivalence of $\infty$-categories 
%\[
%\Fun^{\ast,\times,E}(\Omega,\CC)\simeq \Fun (\Delta^{op}, \CC).
%\]
%\end{cor} 

\begin{prop}\label{prop:DeltaTwoSegal}
Denote by $\tSeg_\Delta(\CC)$ the full subcategory of $\Fun(\Delta^\op, \CC)$ on unital 2-Segal objects. Then the equivalence of the previous proposition descends to an equivalence of $\infty$-categories
\[
\Fun^{alg}(\Delta^\star,\CC)\simeq \tSeg_\Delta(\CC).
\] 
\end{prop} 

\begin{proof}
Let $G\in \Fun^{alg}(\Delta^\star,\CC)$, and consider the diagram
\[
\xymatrix{
[n]\ar[d] & \ar[l]\ar[d][n+m-1] \\
(\{0,1\},\{1,2\},\ldots ,\{n-1,n\}) & ([1],\ldots,\overbrace{[m]}^{j^{th}},\ldots,[1])\ar[l]
}
\]
in $\Delta^\star$. We can expand this diagram to 
\[
\xymatrix{
[n]\ar[d] & \ar[l]\ar[d][n+m-1] \\
(\{0,1\},\{1,2\},\ldots ,\{n-1,n\})\ar[d] & ([1],\ldots,\overbrace{[m]}^{j^{th}},\ldots,[1])\ar[l]\ar[d]\\
\{j-1,j\} & \ar[l][m] 
}
\]
Since the two vertical morphisms in the lower square are sent to projections onto  factors of a product, the lower square is sent to a pullback diagram under $G$. We therefore see that the exterior square is sent to a pullback if and only if the upper square is sent to a  pullback. However, the exterior square is opposite to the diagram 
\[
\xymatrix{
[n]\ar[r] & [n+m-1]\\
[1]\ar[r]_{\{0,m\}}\ar[u]^{\{j-1,j\}} & [m]\ar[u]
}
\]
in $\Delta$, which is precisely the diagram for the 2-Segal conditions when $[m]\neq [0]$, and is the diagram for the unitality condition when $[m]=[0]$. Therefore, we see that $G\in\Fun^\times (\Delta^\star,\CC)$ is in $\Fun^{alg}(\Delta^\star,\CC)$ if and only if the underlying simplicial object is unital 2-Segal. 
\end{proof}

We can summarize our results in the following theorem.

\begin{thm}\label{thm:AlgInSpan}
There is an equivalence of $\infty$-categories
\[
\SAlg(\CC)\simeq \tSeg_\Delta(\CC).
\]
\end{thm}
\section{Calabi-Yau algebras in Spans}

We now  extend the results of the previous section to Calabi-Yau algebras. Throughout this section we set $\Theta:\Tw(\Ass_\CY^\op)\times_{\bbGamma}\bbGamma^\amalg$. We will represent morphisms in $\Theta$ diagrammatically as 
\[
\begin{tikzcd}
 Q\arrow[Subseteq]{r}{}& S & T\arrow[l, "f"']\arrow[d, "\overline{g}"] \\
 P\arrow[Subseteq]{r}{}& S^\prime\arrow[u, "g"] & \arrow[l,"f^\prime"]T^\prime
\end{tikzcd}
\]
where $f,\,f^\prime,\,g,$ and $\overline{g}$ are morphisms in $\Ass_\CY$ (\emph{not} $\Ass_\CY^\op$). 

In general, for a morphism $\diamond\overset{f}{\leftarrow} T$ in $\Ass_\CY$, we will denote the two possible subsets of the image of $\diamond$ in $\bFin_\ast$ by $\emptyset$ and $\{1\}$. 

\subsection{Conditions on functors}

Suppose we are given a functor 
\[
G:\Theta\to \CC
\]
corresponding to a functor $\Tw(\Ass_\CY^\op)\to \CC^\times$ over $\bbGamma$. 

\begin{prop}\label{prop:LambdaPullback}
The functor $G$ defines a functor $\overline{G}: \Ass_\CY^\op\to \Span_{\bbGamma}(\CC^\times)$ if and only if for every simplex $S_0\overset{\phi_1}{\to}S_1\overset{\phi_2}{\to} S_2\to\cdots\to S_n$ in $\Ass_\CY$, and every subset $P\subset S_n^\circ$ the corresponding diagram 
\begin{equation}\label{diag:segalsimpsinLambda}
\begin{tikzcd}[column sep=-20pt]
 &[10pt] & & G(\psi_{n-1},P)\arrow[dll]\arrow[drr] & & & \\
 & G(\phi_n,P)\arrow[dr]\arrow[dl] & &  \cdots\arrow[dr]\arrow[dl] &  & G(\phi_1,\psi_{n-2}^{-1}(P))\arrow[dr]\arrow[dl] & \\
 G (S_n, P) & & G (S_{n-1},\phi_n^{-1}(P))  & \cdots & G(S_1,\psi_{n-2}^{-1}(P)) & & G (S_0,\psi_{n-1}^{-1}(P))  
\end{tikzcd}
\end{equation}
is a limit diagram in $ \CC$, where $\psi_k:=\phi_n\circ \phi_{n-1}\cdots \circ \phi_{n-k} $.
\end{prop}

\begin{proof}
This is, {\itshape mutatis mutandis}, the same as the proof of  \cref{prop:Deltapullbackcond}. Note that if $S_k=\diamond$, then $S_{j}=\diamond$ for all $j\geq k$. 
\end{proof}

\subsubsection{Equivalences}

Suppose that $G:\Theta\to \CC$ represents a co-Calabi-Yau algebra. This means that, for every inert morphism $\phi:S\to T$ in $\Ass\subset \Ass_\CY$, and every $P\subset T^\circ$, 
\begin{itemize}
\item For the source map $\phi\to S$ in $\Tw(\Ass_\CY^\op)$, the induced morphism 
\[
G(\phi, P)\to G(S, \phi^{-1}(P))
\]
is an equivalence
\item For the target map $\phi\to T$ in $\Ass_\CY^\op$, the induced morphism 
\[
G(\phi,P)\to G(T,P) 
\]
is an equivalence. 
\end{itemize}

\begin{lem}\label{lem:inertequivs}
Suppose $G$ represents a co-Calabi-Yau algebra object. Let $\phi:S\to T$ be a morphism in $\Ass$ viewed as an object in $\Tw(\Ass_\CY^\op)$ and let $P\subset T$. 
\begin{enumerate}
\item Let $\psi_2: T\to P$ be the inert morphism is $\Ass$ that acts as the identity on $P$ and sends all other elements to the basepoint. Then the induced morphism 
\[
G(\psi_2\circ \phi,P)\to G(\phi,P)
\] 
is an equivalence.
\item Let $\psi_1:\phi^{-1}(P)\to S$ be morphism in $\Ass$ defined via the inclusion. Then the induced morphism 
\[
G(\psi_2\circ\phi\circ \psi_1,P)\to G(\phi,P)
\]
is an equivalence.
\end{enumerate}
\end{lem}

\begin{proof}
By \cref{prop:LambdaPullback}, the diagrams
\[
\begin{tikzcd}[column sep=-10pt]
 & G(\psi_2\circ\phi,P)\arrow[dl]\arrow[dr] &[-5pt] \\
 G(\psi_2, P)\arrow[dr] & & G(\phi,P)\arrow[dl]\\
  & G(T,P) &
\end{tikzcd}
\]
is a pullback diagram. Since $\psi_2$ is inert in $\Ass$, the morphism 
\begin{align*}
G(\psi_2, P) & \to G(T,P)
\end{align*}
is an equivalence. Therefore, the morphism 
\[
G(\psi_2\circ \phi,P)\to G(\phi,P)
\] 
is an equivalence. 

We now note that the morphism $(\psi_2\circ\phi\circ \chi, P)\to (\phi,P)$ can be factored as 
\[
(\psi_2\circ\phi\circ \chi, P)\to (\psi_2\circ \phi,P)\to (\phi,P).
\] 
Since the second of these morphisms is an equivalence, we need only show that the first is as well. To do this, we write down a composite 
\[
\begin{tikzcd}
 P\arrow[Subseteq]{r}{}& P & \phi^{-1}(P)\arrow[l, "\psi_2\circ \phi\circ \psi_1"']\arrow[d, "\psi_1"] \\
 P\arrow[Subseteq]{r}{}& P\arrow[u, "\id"] & \arrow[l,"\psi_2\circ \phi"]S\arrow[d,"\pi"]\\
 P\arrow[Subseteq]{r}{}& P\arrow[u, "\id"] & \phi^{-1}(P)\arrow[l, "\psi_2\circ \phi\circ \psi_1"']
\end{tikzcd}
\] 
in $\Tw(\Ass_\CY^\op)\times_\Gamma \Gamma^\amalg$, where $\pi$ is the inert morphism projecting $S$ onto the subset $\phi^{-1}(S)$. Since the composite is the identity, it will suffice to show that the bottom square is sent to an equivalence under $G$.  

Denote by $\nu$ the morphism defined by the bottom square. By \cref{prop:LambdaPullback}, we can write down a pullback square
\[
\begin{tikzcd}[column sep=-10pt]
 & G(\psi_2\circ \phi,P)\arrow[dl, "\nu"']\arrow[dr] & \\
 G(\psi_2\circ \phi \circ \psi_1, P)\arrow[dr] & & G(\pi,\phi^{-1}(P))\arrow[dl]\\
  & G(\phi^{-1}(P),\phi^{-1}(P)) &
\end{tikzcd}
\]
The bottom right morphism is the source map of an inert morphism, and thus is an equivalence. Therefore, $\nu$ is also an equivalence.
\end{proof}

\begin{prop}\label{prop:EinCYone}
Suppose that $G$ sends the morphisms from \cref{lem:inertequivs} to equivalences. Let $\mu$ be a morphism 
\[
\begin{tikzcd}
 Q\arrow[Subseteq]{r}{}& S & T\arrow[l, "f"']\arrow[d, "\overline{g}"] \\
 P\arrow[Subseteq]{r}{}& U\arrow[u, "g"] & \arrow[l,"h"]V
\end{tikzcd}
\] 
such that $g|_{P}:P\to Q$ is an isomorphism, $P=g^{-1}(Q)$, and $\overline{g}|_{f^{-1}(Q)}:f^{-1}(Q)\to h^{-1}(P)$ is an isomorphism. Then $G(\mu)$ is an equivalence.
\end{prop}

\begin{proof}
Consider the diagram 
\[
\begin{tikzcd}
Q\arrow[Subseteq]{r}{} & Q & f^{-1}(Q)\arrow[l]\arrow[d,"\subset"]\\
 Q\arrow[Subseteq]{r}{}& S\arrow[u,"\on{proj}"] & T\arrow[l, "f"']\arrow[d, "\overline{g}"] \\
 P\arrow[Subseteq]{r}{}& U\arrow[u, "g"] & \arrow[l,"h"]V
\end{tikzcd}
\]
The top square is a morphism from \cref{lem:inertequivs}, and hence is sent to an equivalence. Moreover, the composite morphism can be decomposed as 
\[
\begin{tikzcd}
Q\arrow[Subseteq]{r}{} & Q & f^{-1}(P)\arrow[l]\arrow[d,"\cong"]\\
P\arrow[Subseteq]{r}{} & P\arrow[u,"\cong"] &\arrow[l] h^{-1}(P)\arrow[d,"\subset"]\\
 P\arrow[Subseteq]{r}{}& U\arrow[u, "\on{proj}"] & \arrow[l,"h"]V
\end{tikzcd}
\]
Since the lower square is sent to an equivalence by \cref{lem:inertequivs} and the upper square is an isomorphism, this composite is sent to an equivalence. Therefore, by the 2-out-of-3 property, $G(\mu)$ is an equivalence.
\end{proof}

\begin{prop}\label{prop:EinCYtwo}
Suppose that $G$ represents a co-Calabi-Yau algebra. Let $\mu$ be a morphism 
\[
\begin{tikzcd}
 \{1\}\arrow[Subseteq]{r}{}& \diamond & T\arrow[l, "f"']\arrow[d, "\overline{g}"] \\
 \{1\}\arrow[Subseteq]{r}{}& \diamond\arrow[u, "\id"] & \arrow[l,"h"]S
\end{tikzcd}
\] 
such that $\overline{g}|_{f^{-1}(\diamond)}:f^{-1}(\diamond)\to g^{-1}(\diamond)$ is an isomorphism. Then $G(\mu)$ is an equivalence.
\end{prop}

\begin{proof}
This is, {\itshape mutatis mutandis}, the same as the proof of \cref{lem:inertequivs} part (2). 
\end{proof}

\begin{defn}
We define the set $E$ of morphisms in $\Theta$ to be the set of all morphisms from \cref{prop:EinCYone} and \cref{prop:EinCYtwo}.
\end{defn}

\subsubsection{Non-degeneracy}

We now consider a morphism $\gamma$  in $\Span_\Gamma(\CC^\times)$ represented by 
\[
X\times X \overset{(\gamma_1,\gamma_2)}{\leftarrow} Y \rightarrow \ast.
\]

\begin{lem}\label{lem:nondegistwoequivs}
The morphism $\gamma$ is non-degenerate in the sense of \cref{defn:nondegenerate} if and only if $\gamma_1$ and $\gamma_2$ are equivalences. 
\end{lem}

\begin{proof}
If $\gamma_1$ and $\gamma_2$ are equivalences, we can define a morphism 
\[
\ast \leftarrow Y \overset{(\gamma_1,\gamma_2)}{\longrightarrow} X\times X
\]
which displays the non-degeneracy of $\gamma$. 

Now suppose that $\gamma$ is non-degenerate, and let $\eta:=(\eta_1,\eta_2)$ be a morphism 
\[
\ast \leftarrow Z \overset{(\eta_1,\eta_2)}{\longrightarrow} X\times X
\]
displaying the non-degeneracy of $\gamma$. Then we have the diagram
\[
	\begin{tikzcd}[column sep=-1ex,cells={nodes={align=center,text width=4em}}]
	& & & Y\arrow[dl, "\nu"']\arrow[dr, "s"] \arrow[bend left, ddrr, "k"] & & &\\
	& & X\arrow[dr, "b"']\arrow[dl, "a"]\arrow[bend left, ddrr, "\ell"] & & Y\times_{X} Y\arrow[dr, "p"']\arrow[dl, "q"] & &\\
	& Z\arrow[dr]\arrow[dl] & & Y\arrow[dr,"\gamma_1"']\arrow[dl,"\gamma_2"] & & Y\arrow[dr,"\gamma_2"']\arrow[dl,"\gamma_1"] &\\
	X & & X & & X & & X
	\end{tikzcd}
\]
where every square is pullback. The left hand pullback must define an equivalence in $\Span(\CC)$, and therefore, the morphism $\ell$ is an equivalence. We thus see that $\gamma_1$ must have a left inverse up to homotopy. Similarly, we see that the morphism $k$ must be an equivalence. By the symmetry of the left-hand pullback square, $q\circ s$ must be an equivalence, and thus , $b\circ \nu$ is an equivalence. However, $\nu$ is a pullback of $\gamma_1$ along an equivalence, and therefore is homotopic to $\gamma_1$. Therefore, we see that $\gamma_1$ has a right inverse up to homotopy, and so, $\gamma_1$ is an equivalence. A similar argument shows that $\gamma_2$ is an equivalence. 
\end{proof}

\begin{const}\label{const:nondegequivs}
Let $G:\Theta\to \CC$ be a functor representing a trace co-algebra in $\Span_\bbGamma(\CC^\times)$. In particular, we have the object 
\[
Y:=G(\{1\}\subset \diamond \leftarrow \langle 2\rangle)
\] 
and the object 
\[
X_n:=G(\{2\}\subset \langle 2\rangle \overset{f}{\leftarrow} \langle n+1\rangle)
\]
where $f(1)=1$ and $f(i)=2$ for all $i\neq 1$. Finally, we have the object 
\[
Z_n:=G(\{1\}\subset \diamond \leftarrow \langle n\rangle
\]
By \ref{prop:LambdaPullback}, we get a pullback diagram 
\[
\begin{tikzcd}
 & \arrow[dl]\arrow[dr]Z_n & \\
 X_n\arrow[dr]&  & Y\arrow[dl] \\
  & (\langle 2\rangle, \{2\}) & 
\end{tikzcd}
\]
By \ref{lem:nondegistwoequivs}, we know that the trace is non-degenerate if and only if the bottom right morphism is an equivalence. From the structure of the pullback diagram, we see that this is equivalent to requiring that the morphism $Z_n\to Z_n$ is an equivalence for all $n$. 
\end{const}

We can summarize the conditions we have worked out in the following corollary

\begin{cor}\label{cor:CYalgconds}
A functor $G:\Theta\to \CC$ defines a Calabi-Yau co-algebra in $\Span_\Gamma(\CC^\times)$ if and only if it satisfies the following conditions:
\begin{enumerate}
\item \label{CYcond:emptysubsets} $G$ sends empty subsets to the terminal object.
\item \label{CYcond:product} $G$ sends $P\subset S\leftarrow T$ together with its projections to $\{i\}\subset S\leftarrow T$ for $i\in P$ to a product diagram.
\item \label{CYcond:Eloc} $G$ sends the morphisms in $E$ to equivalences.
\item \label{CYcond:segal} $G$ sends all diagrams of the form \cref{diag:segalsimpsinLambda} to limit diagrams. 
\item \label{CYcond:nondeg} $G$ sends the morphisms $Z_n\to X_n$ from \ref{const:nondegequivs} to equivalences. 
\end{enumerate}
\end{cor}

\begin{defn}
We define $\SAlg^\CY(\CC)$ to be the full $\infty$-subcategory of $\Fun(\Theta,\CC)$ satisfying the conditions of \cref{cor:CYalgconds}.
\end{defn}

\subsection{The localization}

\begin{defn}
Let $\Lambda^\star$ be the category with objects 
\begin{itemize}
\item finite collections $\{[m_i]\}_{i\in S}$ in $\Delta$ indexed by $S\in \bFin$, and 
\item $\langle n\rangle$ in $\Lambda$, 
\end{itemize}
and morphisms given by 
\begin{enumerate}
\item a morphism $\{[m_i]\}_{i\in S}\to \{[n_j]\}_{j\in T}$ is given by
\begin{itemize}
\item  a morphism $\phi:T\to S$ in $\bFin$, with a chosen linear order on each fiber, and
\item for each $i\in S$, a morphism 
\[
\bigoplus_{j\in \phi^{-1}(i)}[n_j]\to [m_i]
\]
\end{itemize}
\item a morphism $  \langle n \rangle\to \{[m_i]\}_{i\in S}$ is given by 
\begin{itemize}
\item a cyclic order on $S$, and
\item a morphism 
\[
\bigcup\nolimits^S [m_i] \to \langle n\rangle 
\]  
in $\Lambda$. 
\end{itemize}
\item Empty homsets $\{[m_i]\}\to\langle n\rangle$. 
\end{enumerate}
Composition is defined by taking lexicographic linear and cyclic orders. It is well-defined by \cref{lem:LinAndCycOrdSum}.
\end{defn}

\begin{defn}
As in the case of algebra objects, we define a version of $\Theta$ on non-degenerate subsets. Let $\Omega$ be the full subcategory of $\Theta$ on those objects 
\[
Q\subset S\overset{f}{\longleftarrow} T
\]
such that $Q\neq \emptyset$ and $f:T\to S$ is not $\id_\diamond$. 
\end{defn}

\begin{lem}
The is an equivalence of $\infty$-categories 
\[
\Fun^\ast(\Theta, \CC) \simeq \Fun(\Omega,\CC)
\]
Where $\Fun^\ast$ denotes the full subcategory on those functors which send empty subsets to the terminal object of $\CC$. 
\end{lem}

\begin{proof}
This is, {\itshape mutatis mutandis}, the same proof as that of \cref{lem:ForgetAttracting}.
\end{proof}

\begin{figure}
	\begin{tikzpicture}[decoration={
		markings,
		mark=at position 0.5 with {\arrow{>}}}
	]
	\foreach \x in {0, 3, 6, 9}{
		\draw (\x,-0.5) -- (\x, 4.5);
	};
	\path (0,1) coordinate (l2);
	\path (0,3.5) coordinate (l1);
	
	\path (3,3.5) coordinate (lm1);
	\path (3,2) coordinate (lm2);
	\path (3,1) coordinate (lm3);
	\path (3,0) coordinate (lm4);
	
	\path (6,3.5) coordinate (rm1);
	\path (6,2) coordinate (rm2);
	\path (6,1.3) coordinate (rm3);
	\path (6,0.7) coordinate (rm4);
	\path (6,0) coordinate (rm5);
	
	\path (9,4) coordinate (r11);
	\path (9,3) coordinate (r12);
	\path (9,2) coordinate (r2);
	\path (9,1.3) coordinate (r3);
	\path (9,0.7) coordinate (r4);
	
	\foreach \lab in {l1,l2,lm1,lm2,lm3,lm4,rm1,rm2,rm3,rm4,rm5,r11,r12,r2,r3,r4}{
		\draw[fill=black] (\lab) circle (0.05);
	};
	\draw[postaction={decorate}] (rm1) to (lm1); 
	\draw[postaction={decorate}] (lm1) to (l1);
	
	\draw[postaction={decorate}] (rm2) to (lm2);
	\draw[postaction={decorate}] (rm5) to (lm4);
	
	\draw[postaction={decorate}] (r2) to (rm2);
	\draw[postaction={decorate}] (r3) to (rm3);
	\draw[postaction={decorate}] (r4) to (rm4);
	\draw[postaction={decorate}] ($(1.5,0)+(rm5)$) to (rm5);
	\draw[fill=black] ($(1.5,0)+(rm5)$) circle (0.05);
	
	\coordinate (vert) at ($ (-0.7,0)+0.25*(lm2)+0.25*(lm3)+0.25*(lm4)+0.25*(l2)$);
	\draw[postaction={decorate}] (vert) to (l2); 
	\draw (lm2)--(vert) (lm3)--(vert) (lm4)--(vert); 
	
	\coordinate (vert) at ($ (-0.4,0)+0.333*(rm3)+0.333*(rm4)+0.333*(lm3)$);
	\draw[postaction={decorate}] (vert) to (lm3); 
	\draw (rm4)--(vert) (rm3)--(vert); 
	
	\coordinate (vert) at ($ (-0.4,0)+0.333*(r11)+0.333*(r12)+0.333*(rm1)$);
	\draw[postaction={decorate}] (vert) to (rm1); 
	\draw (r11)--(vert) (r12)--(vert);
	
	\draw[red, fill=red] (l1) circle (0.06) (l2) circle (0.06);
	
	\draw[red, fill=red] (lm1) circle (0.06) (lm3) circle (0.06);
	
	%%% interstices
	\path (5.8,4) node (lmint1) {0};
	\path (5.8,2.75) node (lmint2) {1};
	\path (5.8,1.6) node (lmint3) {0};
	\path (5.8,1) node (lmint4) {1};
	\path (5.8,0.3) node (lmint5) {2};
	
	\path (8.8,4.3) node (lint1) {0};
	\path (8.8,3.5) node (lint2) {1};
	\path (8.8,2.4) node (lint3) {2};
	\path (8.8,1.6) node (lint4) {3};
	\path (8.8,1) node (lint5) {4};
	\path (8.8,0) node (lint6) {5};
	
	\begin{scope}[blue]
	\draw[postaction={decorate}] (lmint1) to (lint1);
	\draw[postaction={decorate}] (lmint2) to (lint3);
	\draw[postaction={decorate}] (lmint3) to (lint4);
	\draw[postaction={decorate}] (lmint4) to (lint5);
	\draw[postaction={decorate}] (lmint5) to (lint6);
	\end{scope}
	
	\end{tikzpicture}
	\caption{A pictorial representation of a morphism $\mu$ in $\Omega$, considered as a sequence $T\overset{\overline{g}}{\to} V\overset{h}{\to} U \overset{g}{\to} S$ of morphisms in $\Ass$. The chosen subsets $Q\subset S$ and $P\subset U$ are marked in red, and the induced morphism $\mathcal{L}(\mu)$ is drawn in blue. Note that, unlike in the analogous \cref{fig:DeltaLocMorphs}, the source of $\mathcal{L}(\mu)$ is the ordinal sum  $\bigoplus_{i\in P} O(f^{-1}(i))$, owing to the presence interstitial trees with roots not in $P$.}\label{fig:LambdaLocMorphs}
\end{figure}
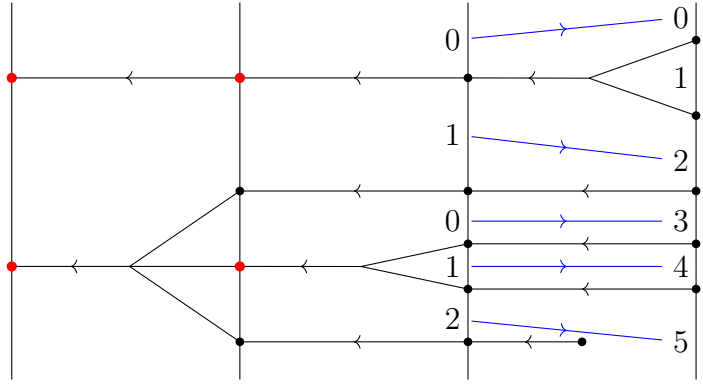

\begin{const}
We define a functor $\Omega\to \Lambda^\star$ as follows. Let 
\[
P\subset S\overset{f}{\leftarrow} T
\]
be an object in $\Omega$ with $f$ a morphism in $\Ass$. We send this object to the collection 
\[
\left\lbrace O\left(f^{-1}(i)\right) \right\rbrace_i\in P.
\] 
Let 
\[
\{\diamond\}\subset \diamond \overset{f}{\leftarrow} S
\]
be an object in $\Omega$. Then we send this object to 
\[
D(f^{-1}(\diamond))\in \Lambda.
\]

To define $\mathcal{L}$ on morphisms, we proceed by cases: 
\begin{enumerate}
\item Suppose we have a diagram 
\[
\begin{tikzcd}
 Q\arrow[Subseteq]{r}{}& S & T\arrow[l, "f"']\arrow[d, "\overline{g}"] \\
 P\arrow[Subseteq]{r}{}& U\arrow[u, "g"] & \arrow[l,"h"] V
\end{tikzcd}
\]
representing a morphism $\mu$ in $\Omega$, where all of the objects are in $\Ass\subset \Ass_\CY$. $\mathcal{L}(\mu)$ will be given by a morphism $\phi_\mu$ in $\Fin_\ast$ and a set of morphisms $\{\psi_i\}_{i\in Q}$ in $\Delta$. The morphism $\phi_\mu$ we take to be the restriction of $g$ to $P\subset U^\circ$. Fixing $i\in Q$, we see that $\overline{g}$ restricts to a morphism $\overline{g}_i: f^{-1}(i)\to h^{-1}(g^{-1}(i))$ of linearly ordered sets. This can be rewritten  as 
\[
\overline{g}_i:f^{-1}(i)\to\bigoplus_{j\in g^{-1}(i)} h^{-1}(j) 
\] 
It therefore induces a morphism 
\[
\bigstar_{j\in\overline{g}^{-1}(i)} O(h^{-1}(i))\to O(f^{-1}(i))
\]
We then define $\psi_i$ to be the composite
\[
\bigoplus_{j\in\overline{g}^{-1}(i)\cap P} O(h^{-1}(i))\to\bigstar_{j\in\overline{g}^{-1}(i)} O(h^{-1}(i))\to O(f^{-1}(i))
\]
See \cref{fig:LambdaLocMorphs} for a pictorial representation.
\item Suppose we have a diagram
\[
\begin{tikzcd}
 \{1\}\arrow[Subseteq]{r}{}& \diamond & T\arrow[l, "f"']\arrow[d, "\overline{g}"] \\
 \{1\}\arrow[Subseteq]{r}{}& \diamond\arrow[u, "\id"] & \arrow[l,"h"] V
\end{tikzcd}
\]
representing a morphism $\mu$ in $\Omega$ with $T, U\in \Ass$. Then $\mathcal{L}(\mu)$ will be given by a morphism $\psi:D(h^{-1}(\diamond))\to D(f^{-1}(\diamond))$. The morphism $\overline{g}$ restricts to a morphism of cyclically ordered sets 
\[
\overline{g}_\diamond:f^{-1}(\diamond)\to h^{-1}(\diamond)
\]
we therefore define $\psi$ to be $D(\overline{g}_\diamond)$.
\item Suppose we have a diagram 
\[
\begin{tikzcd}
 \{1\}\arrow[Subseteq]{r}{}& \diamond & T\arrow[l, "f"']\arrow[d, "\overline{g}"] \\
 P\arrow[Subseteq]{r}{}&U \arrow[u, "g"] & \arrow[l,"h"] V
\end{tikzcd}
\]
representing a morphism $\mu$ in $\Omega$, where all objects except $\diamond$ are in $\Ass$. The morphism $\mathcal{L}(\mu)$ will be given by a cyclic order on $P$ and a morphism $\psi: \bigcup^S O(f^{-1}(i))\to D(f^{-1}(\diamond))$. The cyclic order on $P$ is induced by the cyclic order on $g^{-1}(\diamond)\supset P$. The morphism $\overline{g}$ restricts to a morphism 
\[
\overline{g}_\diamond: f^{-1}(\diamond)\to (g\circ h)^{-1}(\diamond)
\]   
of cyclically ordered sets. Passing through  $D$ gives a morphism 
\[
D(\overline{g}_\diamond): D((g\circ h)^{-1}(\diamond))\to D(f^{-1}(\diamond)). 
\]
Choosing any linear order on $g^{-1}(\diamond)$ compatible with the cyclic order we can write $D(\overline{g}_\diamond)$ as 
\[
C(O(\bigoplus_{i\in g^{-1}(\diamond)}h^{-1}(i))) = D(K(\bigoplus_{i\in g^{-1}(\diamond)}h^{-1}(i)) \to D(f^{-1}(\diamond)
\]
We then have the canonical morphism
\[
K(\bigoplus_{i\in g^{-1}(\diamond)}O(h^{-1}(i)))\to C(\bigstar_{i\in g^{-1}(\diamond)}O(h^{-1}(i)))= C(O(\bigoplus_{i\in g^{-1}(\diamond)}h^{-1}(i)))
\]
And so we define $\psi$ to be the composite 
\[
K(\bigoplus_{i\in P}O(h^{-1}(i)))\to K(\bigoplus_{i\in g^{-1}(\diamond)}O(h^{-1}(i)))\to C(O(\bigoplus_{i\in g^{-1}(\diamond)}h^{-1}(i))) \to D(f^{-1}(\diamond) 
\]
See \cref{fig:CircleDecomp} for a pictorial representation.
\end{enumerate}
\end{const}
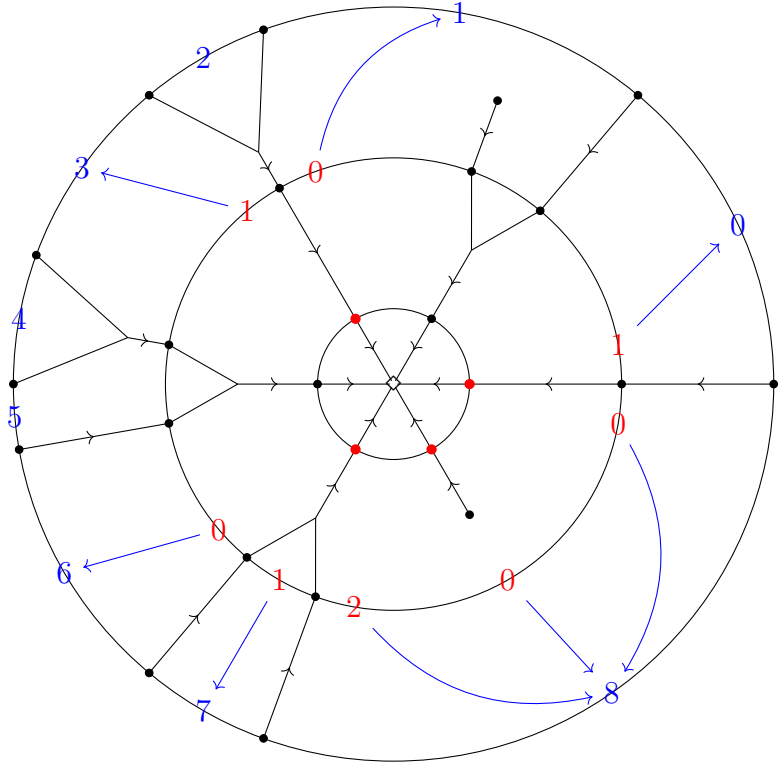
\begin{figure}[htb]
	\begin{tikzpicture}[decoration={
	markings,
	mark=at position 0.5 with {\arrow{>}}}
]

%%Ins
\path (0,0) node (c) {$\diamond$};
\draw (0,0) circle (1);
\draw (0,0) circle (3);
\draw (0,0) circle (5);
\foreach \x/\lab in {0/in1,60/in2,120/in3,180/in4,240/in5,300/in6}{
\path (\x:1) coordinate (\lab);
\draw[fill=black] (\x:1) circle (0.05);
\draw[postaction={decorate}] (\x:1) to (\x:0.05);
};

%%mids
\path (0:3) coordinate (mid1) {};
\draw[fill=black] (0:3) circle (0.05);
\draw[postaction={decorate}] (0:3) to (in1.center);

\foreach \dif/\lab in {-10/mid2,10/mid3}{
\path (60+\dif:3) coordinate (\lab) {};
\draw[fill=black] (60+\dif:3) circle (0.05);
};
\coordinate (vert223) at ($ 0.33*(mid2)+0.33*(mid3)+0.1*(in2)$);
\draw[postaction={decorate}] (vert223) to (in2); 
\draw (mid2)--(vert223) (mid3)--(vert223);

\path (120:3) coordinate (mid4);
\draw[postaction={decorate}] (mid4) to (in3);
\draw[fill=black] (mid4) circle (0.05);

\foreach \dif/\lab in {-10/mid5,10/mid6}{
	\path (180+\dif:3) coordinate (\lab) {};
	\draw[fill=black] (180+\dif:3) circle (0.05);
};
\coordinate (vert456) at ($ 0.33*(mid5)+0.33*(mid6)+0.1*(in4)$);
\draw[postaction={decorate}] (vert456) to (in4); 
\draw (mid6)--(vert456) (mid5)--(vert456);

\foreach \dif/\lab in {-10/mid7,10/mid8}{
	\path (240+\dif:3) coordinate (\lab) {};
	\draw[fill=black] (240+\dif:3) circle (0.05);
};
\coordinate (vert578) at ($ 0.33*(mid8)+0.33*(mid7)+0.1*(in5)$);
\draw[postaction={decorate}] (vert578) to (in5); 
\draw (mid7)--(vert578) (mid8)--(vert578);

\draw[postaction={decorate}] (300:2)--(in6);
\draw[fill=black] (300:2) circle (0.05);

%% outs 
\foreach \lab/\nom/\oset in {mid1/out1/0,mid2/out2/0,mid4/out4/-10, mid4/out5/10,mid5/out6/-10,mid5/out7/10,mid6/out8/0,mid7/out9/0,mid8/out10/0}{
\pgfextractangle{\myangle}{c}{\lab};
\path (\myangle+\oset:5) coordinate (\nom);
\draw[fill=black] (\nom) circle (0.05);
};

\foreach \m/\j/\k in {mid4/out4/out5, mid5/out6/out7}{
\coordinate (vert) at ($ 0.33*(\k)+0.33*(\j)+0.1*(\m)$);
\draw[postaction={decorate}] (vert) to (\m); 
\draw (\j)--(vert) (vert)--(\k);
};
\draw[postaction={decorate}] (out1) to (mid1);
\draw[postaction={decorate}] (out2) to (mid2);
\draw[postaction={decorate}] (out8) to (mid6);
\draw[postaction={decorate}] (out9) to (mid7);
\draw[postaction={decorate}] (out10) to (mid8);
\draw[postaction={decorate}] (70:4) to (mid3);
\draw[fill=black] (70:4)  circle (0.05);

\begin{scope}[red]
\foreach \lab in {in1,in3,in5,in6}{
\draw[fill=red] (\lab) circle (0.06);
};
\end{scope}

\begin{scope}[blue]
\foreach \lab/\nom/\wor/\nam in {out1/out2/int1/0,out2/out4/int2/1,out4/out5/int3/2,out5/out6/int4/3,out6/out7/int5/4,out7/out8/int6/5,out8/out9/int7/6,out9/out10/int8/7}{
\pgfextractangle{\myangle}{c}{\nom};
\pgfextractangle{\myother}{c}{\lab};
\path (0.5*\myangle+0.5*\myother:5) node (\wor) {\nam};
};
\pgfextractangle{\myother}{c}{out10};
\path (180+0.5*\myother:5) node (int9) {8};
\end{scope}
\begin{scope}[red]
\path (-10:3) node (mint1) {0};
\path (10:3) node (mint2) {1};
\path (110:3) node (mint3) {0};
\path (130:3) node (mint4) {1};
\path (220:3) node (mint5) {0};
\path (240:3) node (mint6) {1};
\path (260:3) node (mint7) {2};
\path (300:3) node (mint8) {0};
\end{scope}

\begin{scope}[blue]
\draw[->] (mint1) to[bend left] (int9);
\draw[->] (mint2) to (int1);
\draw[->] (mint3) to[bend left] (int2);
\draw[->] (mint4) to (int4);
\draw[->] (mint5) to (int7);
\draw[->] (mint6) to (int8);
\draw[->] (mint7) to[bend right] (int9);
\draw[->] (mint8) to (int9);
\end{scope}
\end{tikzpicture}
	\caption{A morphism in $\Omega$ represented as a composite of three morphisms in $\Ass_\CY$, $T\overset{\overline{g}}{\to} V\overset{h}{\to} U\overset{g}{\to} \diamond$. The chosen subset $P\subset U$ is marked by red points. The corresponding interstice sets $I(h^{-1}(i))$ are written in red numbers, and the set $D(f^{-1}(\diamond))$ in blue numbers. The induced morphism $\mathcal{L}(\mu):\bigcup^P I(h^{-1}(i))\to D(f^{-1}(\diamond))$ is drawn in blue. Note that the unmarked points in $U$ are the reason that we do not necessarily get a morphism $C(\bigstar_{i\in P} I(h^{-1}(i))\to D(f^{-1}(\diamond))$.}\label{fig:CircleDecomp}
\end{figure}

\begin{defn}
Let $M\in \Lambda^\star$. We denote by $\Omega_M^E$ the subcategory of the weak fiber $\Omega_M$ whose morphisms are morphisms in $E$. 
\end{defn}

\begin{prop}\label{prop:LambdaWeakFibers}
For every $M$ in $\Lambda^\star$, there is an initial element in $\Omega_M^E$.
\end{prop}

\begin{proof}
We will complete the proof in two cases: 

Suppose first that $M=\{[m_i]\}_{i\in P}$. Then the weak fiber only involves morphisms in $\Ass\subset \Ass_\CY$. We define a set 
\[
T:= \coprod_{i\in P} \mathbb{I}([m_i]) 
\]
and a morphism $f_M: T\to P$ by setting $f_M(\mathbb{I}([m_i]))=i$. The canonical isomorphisms 
\[
\eta_i:O(\mathbb{I}([m_i]))\cong [m_i]
\]
equip $P\subset P \overset{f_M}{\longleftarrow} T$ with the structure of an object of $\Omega_M^E$. 
Given an element 
\[
P\subset U\overset{f}{\leftarrow} V 
\]
and an isomorphism $\phi_i: O(f^{-1}(i))\cong [m_i]$, we define a unique morphism $\mu$ in $\Omega_M^E$ given by
\[
\begin{tikzcd}
 P\arrow[Subseteq]{r}{}& P & T\arrow[l, "f_M"']\arrow[d, "\overline{g}"] \\
 P\arrow[Subseteq]{r}{}& U\arrow[u, "g"] & \arrow[l,"f"] V
\end{tikzcd}
\]
 as follows. Since this must be a morphism in $E$, we see that $g$ must map $P$ identically to $P$, and send $U^\circ \setminus P$ to the basepoint. On fibers, we consider the isomorphisms
\[
\eta_i^{-1}\circ \phi_i: O(f^{-1}(i))\to O(I[m_i])
\]
Since $O$ is fully faithful, this lifts to a unique isomorphism $I(\phi_i):I([m_i])\cong f^{-1}(i)$. We therefore see that $\overline{g}$ must be the coproduct of these morphisms if $\mu$ is to be a morphism in the weak fiber. It is immediate that this does, indeed, define a morphism in $\Omega_M^E$. 

Now suppose instead $M=\langle m\rangle$. We define $f_M: D(\langle m\rangle) \to \diamond$ to be the morphism with $f_M^{-1}(\diamond)=D(\langle m\rangle)$. Since $D$ is an equivalence,  we choose the isomorphism 
\[
\eta: D^2(\langle m \rangle)\cong \langle m\rangle
\] 
Suppose given another element 
\[
\{1\}\subset \diamond \overset{f}{\longleftarrow} T
\]
with $\phi: D(f^{-1}(\diamond))\cong \langle n\rangle$ in the weak fiber. We define a unique morphism $\mu\in \Omega_M^E$ given by
\[
\begin{tikzcd}
 \{1\}\arrow[Subseteq]{r}{}& \diamond & D(\langle m\rangle \arrow[l, "f_M"']\arrow[d, "\overline{g}"] \\
 \{1\}\arrow[Subseteq]{r}{}& \diamond\arrow[u, "g"] & \arrow[l,"f"] T
\end{tikzcd}
\]
as follows. The morphism $g$ must be the identity, so we need only define $\overline{g}$. The condition that $\mu$ be in the weak fiber implies that $\eta_i\circ D(\overline{g}|_{D(\langle m\rangle)})=\phi_i$, i.e. $D(\overline{g}|_{D(\langle m\rangle)})=\eta_i^{-1}\circ \phi_i$. However, since $D$ is fully faithful, this condition defines a unique isomorphism $D(\langle m\rangle )\cong f^{-1}(\diamond)$, determining $\overline{g}$, and thus $\mu$, uniquely. 
\end{proof}

\begin{prop}\label{prop:LambdaLaxFibOverAss}
Suppose given an object $M=\{[m_i]\}_{i\in P}$ in $\Lambda^\star$, an object
\[
Z:= \left\lbrace Q\subset S\overset{f_Z}{\leftarrow} T \right\rbrace
\]
in $\Omega$, and a morphism 
\[
(\phi, \{\gamma_i\}_{i\in Q}):\mathcal{L}(Z)\to M
\]
in $\Lambda^\star$. Then there is an element $X_{M,Z}$ in $\Omega_M^E$ and a morphism $\Phi:Z\to X_{M,Z}$ in $\Omega$ covering $(\phi,\{\gamma_i\}_{i\in Q})$ such that, for any other morphism $\Psi:Z\to X$ covering $(\phi,\{\gamma_i\}_{i\in Q})$, there is a unique morphism $\tau: X_{M,Z}\to X$ which makes the diagram 
\[
\begin{tikzcd}
 & Z\arrow[dr, "\Psi"]\arrow[dl,"\Phi"'] & \\
 X_{M,Z}\arrow[rr, "\tau"'] & & X
\end{tikzcd}
\]
commute.
\end{prop}

\begin{proof}
There are two cases to consider, corresponding to whether or not $S=\diamond$. 

\phantomsection
\label{proof:CaseOne}
{\itshape Case 1:}  First suppose $S\in \Ass$. In this case, we construct $X_{M,Z}$ as follows. Let 
\[
P\subset P \overset{f_M}{\longleftarrow} U
\]
be the object constructed in \cref{prop:LambdaLaxFibOverAss}. Then, in particular, $\phi:P\to Q\subset S$. 

For each $i\in Q$, we have a morphism 
\[
\gamma_i: \bigoplus_{j\in \phi^{-1}(i)}[m_i]\to O(f_Z^{-1}(i))
\]
For each $j\in \phi^{-1}(i)$ denote by $\gamma_i([m_j])$ the smallest subinterval of $O(f_Z^{-1}(i))$ containing the image of $[m_j]$ under $\gamma_i$. Then $\gamma_i|_{[m_j]}\to \gamma_i([m_j])$ preserves boundary, and thus corresponds to a map $\overline{g}_{j}:I(\gamma_i([m_j]))\to I([m_j])$ of linearly ordered sets. Moreover, $\overline{g}_j$ fits into a commutative diagram 
\[
\begin{tikzcd}
S & I(\gamma_i([m_j]))\arrow[l,"f_{Z}"']\arrow[d,"\overline{g}_j"]\\
P\arrow[u,"\phi"] & I([m_j])\arrow[l,"f_M"']
\end{tikzcd}
\]
in $\Ass$. We here use the identification of  $I(\gamma_i([m_j]))$ with a subset of $T$. 

Since, by definition, $U=\coprod_{j\in P} I([m_j])$, we can then write down a commutative diagram 
\begin{equation}\label{diag:CommXZM}
\begin{tikzcd}
S & \coprod_{i,j}I(\gamma_i([m_j]))\arrow[l,"f_{Z}"']\arrow[d,"\coprod_{i,j}\overline{g}_j"]\\
P\arrow[u,"\phi"] & \coprod_{j\in P} I([m_j])\arrow[l,"f_M"']
\end{tikzcd}
\end{equation}
in $\Ass$. 

For each $i\in Q$, this restricts to a diagram of ordered sets 
\[
\begin{tikzcd}
\{i\} & \coprod_{j}I(\gamma_i([m_j]))\arrow[l,"f_{Z}"']\arrow[d,"\coprod_{j}\overline{g}_j"]\\
\phi^{-1}(i)\arrow[u,"\phi"] & \coprod_{j\in\phi^{-1}(i)} I([m_j])\arrow[l,"f_M"']
\end{tikzcd}
\]
We denote $L_i:=f_Z^{-1}(i)\setminus\coprod_{j\in \phi^{-1}(i)}I(\gamma_i([m_j])$, and proceed as follows.
\begin{itemize}
\item For $p, p+1$ in $\phi^{-1}(i)$, if there is at least one $k\in L_i$ such that
\[
I(\gamma_i([m_{p}]))<k< I(\gamma_i([m_{p+1}]))
\]
 we define a new element $r_p$ and append it to $\phi^{-1}(i)$ between $p$ and $p+1$.
\item If there exists $k\in L_i$ such that 
\[
k< I(\gamma_i([m_{p}]))
\]
for all $p\in \phi^{-1}(i)$, then we append a new minimal element $r_{min}$ to $\phi^{-1}(i)$. 
\item If there exists $k\in L_i$ such that 
\[
I(\gamma_i([m_{p}]))< k
\]
for all $p\in \phi^{-1}(i)$, then we append a new maximal element $ $ to $\phi^{-1}(i)$.
\end{itemize} 
Call the resulting set $W_i\supset \phi^{-1}(i)$. We then set 
\[
R_i:= U\amalg L_i
\]
and define $f_{i}: R_i \to W_i$ to act as $f_{M}$ on $U$ and on $L_i$ to send
\begin{itemize}
\item $k\mapsto r_p$ if 
\[
I(\gamma_i([m_{p}]))<k< I(\gamma_i([m_{p+1}]))
\]
\item $k\mapsto r_{min}$ if 
\[
k< I(\gamma_i([m_{p}]))
\]
for all $p\in \phi^{-1}(i)$
\item $k\mapsto r_{max}$ if 
\[
I(\gamma_i([m_{p}]))< k
\]
for all $p\in \phi^{-1}(i)$
\end{itemize}
We make $f_i$ into a morphism in $\Ass$ by taking the linear order induced by $L_i$ on the fibers over the $r_{p}$, $r_{min}$ and $r_{max}$. We then define 
\[
\overline{g}^i: f_Z^{-1}(i)\to R_i
\]
to act as $\coprod_{j\in \phi^{-1}(i)} \overline{g}_j$ on $\coprod_{j\in \phi^{-1}(i)}  I(\gamma_i([m_j]))$ and as the identity on $L_i$. We further define $\phi_i: W_i\to \{i\}$ to send every element to $i$. We thus have a commutative diagram 
\[
\begin{tikzcd}
\{i\}\arrow[Subseteq]{r}{} &\{i\} & \coprod_{j}I(\gamma_i([m_j]))\arrow[l,"f_{Z}"']\arrow[d,"\overline{g}^i"]\\
P\cap \phi^{-1}(i)\arrow[Subseteq]{r}{} & W_i\arrow[u,"\phi_i"] & R_i\arrow[l,"f_i"']
\end{tikzcd}
\]
in $\Ass$, which covers the morphism $\gamma_i:\bigoplus_{j\in \phi^{-1}(i)} [m_j] \to O(f_Z^{-1}(i))$. Taking the coproduct over $i\in \on{Im}(\phi)$ gives us a morphism
\[
\begin{tikzcd}
\on{Im}(\phi)\arrow[Subseteq]{r}{} &\on{Im}(\phi) & f_Z^{-1}(\on{Im}(\phi))\arrow[l,"f_{Z}"']\arrow[d,"\coprod_{i}\overline{g}^i"]\\
P\cap \phi^{-1}(i)\arrow[Subseteq]{r}{} & \coprod_{i}W_i\arrow[u,"\coprod_{i}\phi_i"] & \coprod_{i}R_i\arrow[l,"\coprod_{i}f_i"']
\end{tikzcd}
\]

Finally, we set
\[
W=(\coprod_{i\in \on{Im}(\phi)}W_i) \amalg (S\setminus \on{Im}(\phi))
\]
and 
\[
R=(\coprod_{i\in \on{Im}(\phi)}R_i) \amalg (T\setminus f_{Z}^{-1}(\on{Im}(\phi)))
\]
We then define morphisms:
\begin{itemize}
\item $g:W\to S$ to act as $\coprod_{i\in \on{Im}(\phi)} \phi_i$ on $\coprod_{i\in \on{Im}(\phi)} W_i$ and as the identity otherwise.
\item $f_{M,Z}: R\to W$ to act as $f_i$ on $R_i$ and as $f_Z$ on $T\setminus f_Z^{-1}(\on{Im}(\phi))$. 
\item $\overline{g}:T\to R$ to act as $\coprod_{i\in \on{Im}(\phi)} \overline{g}^{i}$ on $f_Z^{-1}(\on{Im}(\phi))$ and the identity elsewhere. 
\end{itemize}
By construction, this defines a commutative diagram
\begin{equation}\label{diag:MorphPhi}
\begin{tikzcd}
Q\arrow[Subseteq]{r}{}  & S & T\arrow[l,"f_{Z}"']\arrow[d,"\coprod_{i}\overline{g}"]\\
P\arrow[Subseteq]{r}{} & W\arrow[u,"g"] & R\arrow[l,"f_{M,Z}"']
\end{tikzcd}
\end{equation}
in $\Ass$, covering $(\phi,\{\gamma_i\})$, and the bottom row is in $\Omega_M$. We therefore define $X_{M,Z}$ to be the bottom row, and $\Phi$ to be the morphism defined by the diagram \eqref{diag:MorphPhi}. 

To check the remaining universal property, we 
let 
\[
P\subset A\overset{f_X}{\leftarrow} B
\]
and $\beta_i: O(f_X^{-1}(i))\cong [m_i]$ be another element in $\Omega_M^E$, and let $\nu$ be a morphism 
\begin{equation*}
\begin{tikzcd}
 Q\arrow[Subseteq]{r}{}& S & T\arrow[l, "f_Z"']\arrow[d, "\overline{\rho}"] \\
 P\arrow[Subseteq]{r}{}& A\arrow[u, "\rho"] & \arrow[l,"f_{X}"] B
\end{tikzcd}
\end{equation*}
covering $(\phi, \{\gamma_i\}_{i\in Q})$.

For each $i\in \on{Im}(\phi)$, the identity on $P$ and the condition nothing be sent to the basepoint uniquely determines a map of ordered sets 
\[
\zeta_i:\rho^{-1}(i)\to W_i.
\]
Moreover, the $\zeta_i$ together with the restriction of $\rho$ to $A\setminus \rho^{-1}(\on{Im}(\phi))$ uniquely determines a map 
\[
\zeta: A\to W
\]
such that the diagram 
\[
\begin{tikzcd}
A\arrow[rr,"\rho"]\arrow[dr,"\zeta"] & &  S \\
 & W\arrow[ur, "g"] &
\end{tikzcd}
\]
commutes. Note that $\zeta|_P$ induces the identity $P\to P$. 

Moreover, for each $i\in \on{Im}(\phi)$ the isomorphisms $I(\beta_i)$ on $I([m_j])$ and restriction $\overline{\rho}_i:L_i\to f_X^{-1}(\rho^{-1}(i))$  uniquely determine a map 
\[
\overline{\zeta}_i:R_i \to f_X^{-1}(\rho^{-1}(i)).
\]
These, together with the restriction of $\overline{\rho}$ to $T\setminus f_Z^{-1}(\on{Im}(\phi))$ uniquely determine a morphism 
\[
\overline{\zeta} R\to B
\]
such that the diagram 
\[
\begin{tikzcd}
B & & T\arrow[dl, "\overline{g}"]\arrow[ll,"\overline{\rho}"'] \\
 & R\arrow[ul,"\overline{\zeta}"] & 
\end{tikzcd}
\]
commutes, and the restriction of $\overline{\zeta}$ to $f_X^{-1}(P)$ is the isomorphism $\coprod_{i}I(\beta_i)$. 

We therefore have constructed a unique morphism 
\[
(\zeta,\overline{\zeta}):X_{Z,M} \to X
\]
in $\Omega_M^E$ such that the diagram 
\[
\begin{tikzcd}
 & Z\arrow[dr, "\Psi"]\arrow[dl,"\Phi"'] & \\
 X_{M,Z}\arrow[rr, "{(\zeta,\overline{\zeta})}"'] & & X
\end{tikzcd}
\]
commutes. 

{\itshape Case 2:} Now suppose that $S=\diamond$. Then $\phi$ is completely determined by a cyclic order on $P$, and $\gamma$ is a morphism 
\[
\gamma: \bigcup\nolimits^S [m_i] \to D(f_Z^{-1}(\diamond)).
\]
We note that, given any morphism 
\[
\begin{tikzcd}
 \{1\}\arrow[Subseteq]{r}{}& \diamond & V\arrow[l, "p"']\arrow[d, "\overline{g}"] \\
 A\arrow[Subseteq]{r}{}& B\arrow[u, "g"] & \arrow[l,"\ell"]C
\end{tikzcd}
\]
a choice of linear order on $g^{-1}(\diamond)$ compatible with the cyclic order uniquely determines a factorization
\[
\begin{tikzcd}
 \{1\}\arrow[Subseteq]{r}{}& \diamond & V\arrow[l, "p"']\arrow[d, "\id_V"] \\
  \{1\}\arrow[Subseteq]{r}{}& \{1\}\arrow[u,] & V\arrow[l, "p"']\arrow[d, "\overline{g}"] \\
 A\arrow[Subseteq]{r}{}& B\arrow[u, "g"] & \arrow[l,"\ell"]C
\end{tikzcd}
\]
Similarly, given a morphism $(\psi, \eta): \langle n\rangle \to \{[n_i]\}_{i\in S}$, a choice of linear order on $S$ compatible with the cyclic order uniquely determines a factorization 
\[
\langle n\rangle \to \{[n]\} \to \{[n_i]\}.
\]
We can therefore choose a linear order on $P$ and define $Y$ to be the object 
\[
\{1\} \subset \{1\} \overset{f_Z}{\leftarrow} T.
\]
Then take $(\phi_Y,\gamma_Y)$ to be the unique morphism yielding a factorization
\[
\mathcal{L}(\phi,\gamma): D(f^{-1}_{Z}(\diamond))\to K(f^{-1}_Z(\diamond))\overset{(\phi_Y,\gamma_Y)}{\longrightarrow} \{[m_i]\}_{i\in P} 
\]
We can then construct $X_{M,Y}$ as in \hyperref[proof:CaseOne]{case 1}. It is immediate that 
\[
\begin{tikzcd}
 \{1\}\arrow[Subseteq]{r}{}& \diamond & T\arrow[l, "f_Z"']\arrow[d, "\id_T"] \\
  \{1\}\arrow[Subseteq]{r}{}& \{1\}\arrow[u,] &T\arrow[l, "f_Z"']\arrow[d, "\overline{g}"] \\
 P\arrow[Subseteq]{r}{}& W_Y\arrow[u, "g"] & \arrow[l,"f_{M,Y}"]R_Y
\end{tikzcd}
\] 
defines a morphism $\Phi$ in $\Omega$ covering $(\phi,\gamma)$. 

Now suppose given any other morphism $\Psi=(\psi,\overline{\psi}): Z\to X$ covering $(\phi, \gamma)$. A choice of linear order on $\psi^{-1}(\diamond)$ compatible with the chosen linear order on $P$ uniquely factors $\Psi$ through $Y$. We therefore get a morphism $\tau: X_{M,Y}\to X$ such that the diagram
\[
\begin{tikzcd}
 & Z\arrow[dr, "\Psi"]\arrow[dl,"\Phi"'] & \\
 X_{M,Y}\arrow[rr, "\tau"'] & & X
\end{tikzcd}
\]
 commutes. 

To see that this morphism is unique, suppose that $(\xi,\overline{\xi}), (\zeta,\overline{\zeta}): X_{M,Y}\to X$ are two such morphisms. Then, choosing a linear order on $\psi^{-1}(\diamond)$ compatible with the chose linear order on $P$ uniquely factors the diagram as
\[
\begin{tikzcd}
 & Z\arrow[d]\arrow[ddr, "\Psi"]\arrow[ddl,"\Phi"'] & \\
 & Y\arrow[dl]\arrow[dr] & \\
 X_{M,Y}\arrow[rr,shift left=0.75ex]\arrow[rr, shift right=0.75ex, "{(\xi,\overline{\xi}), (\zeta,\overline{\zeta})}"'] & & X
\end{tikzcd}
\]
But, by \hyperref[proof:CaseOne]{case 1}, there is a unique morphism making the bottom triangle commute. Therefore, $(\xi,\overline{\xi})=(\zeta,\overline{\zeta})$, proving the proposition.  
\end{proof}

\begin{prop}\label{prop:LambdaLaxFibOverLambda}
Suppose given an object $M=\langle m\rangle$ in $\Lambda^\star$, an object
\[
Z:= \left\lbrace Q\subset S\overset{f_Z}{\leftarrow} T \right\rbrace
\]
in $\Omega$, and a morphism 
\[
(\phi, \{\gamma_i\}_{i\in Q}):\mathcal{L}(Z)\to M
\]
in $\Lambda^\star$. Then there is an element $X_{M,Z}$ in $\Omega_M^E$ and a morphism $\Phi:Z\to X_{M,Z}$ in $\Omega$ covering $(\phi,\{\gamma_i\}_{i\in Q})$ such that, for any other morphism $\Psi:Z\to X$ covering $(\phi,\{\gamma_i\}_{i\in Q})$, there is a unique morphism $\tau: X_{M,Z}\to X$ which makes the diagram 
\[
\begin{tikzcd}
 & Z\arrow[dr, "\Psi"]\arrow[dl,"\Phi"'] & \\
 X_{M,Z}\arrow[rr, "\tau"'] & & X
\end{tikzcd}
\]
commute.
\end{prop}

\begin{proof}
We first note that $S=\diamond$, since otherwise no such morphism $(\phi,\{\gamma_i\}_{i\in Q})$ can exist. Consequently, $\phi=\id_\diamond$, and $\gamma$ is a morphism of cyclically ordered sets $\langle m\rangle \to  D(f_Z^{-1}(\diamond))$. We can therefore take $X_{Z,M}$ to be the object
\[
\{1\}\subset \diamond \overset{f_M}{\leftarrow} D(\langle m\rangle)
\]
constructed in the proof of \cref{prop:LambdaWeakFibers}. We then get a commutative diagram
\[
\begin{tikzcd}
 \{1\}\arrow[Subseteq]{r}{}& \diamond & T\arrow[l, "f_Z"']\arrow[d, "\overline{g}"] \\
 \{1\}\arrow[Subseteq]{r}{}& \diamond\arrow[u, "\id"] & \arrow[l,"f_M"] D(\langle m\rangle)\amalg (T\setminus f_Z^{-1}(\diamond))
\end{tikzcd}
\]
where $\overline{g}$ acts as $D(\gamma)$ on $f_Z^{-1}(\diamond)$ and the identity on $T\setminus f_Z^{-1}(\diamond)$. This morphism in $\Omega$ clearly covers $(\id,\gamma)$. 

Given $X\in \Omega_M^E$ and $\Psi:Z\to X$, represented by a diagram 
\[
\begin{tikzcd}
 \{1\}\arrow[Subseteq]{r}{}& \diamond & T\arrow[l, "f_Z"']\arrow[d, "\overline{\ell}"] \\
 \{1\}\arrow[Subseteq]{r}{}& \diamond\arrow[u, "\id"] & \arrow[l,"f_X"] A
\end{tikzcd}
\]
by \ref{prop:LambdaWeakFibers} that there is a unique morphism 
\[
\begin{tikzcd}
 \{1\}\arrow[Subseteq]{r}{}& \diamond & D(\langle m\rangle)\arrow[l, "f_M"']\arrow[d, "\overline{h}"] \\
 \{1\}\arrow[Subseteq]{r}{}& \diamond\arrow[u, "\id"] & \arrow[l,"f_X"] A
\end{tikzcd}
\] 
in $\Omega_M^E$.  Via the restriction of $\overline{\ell}$ to $T\setminus f^{-1}_Z(\diamond)$, this extends to a morphism 
\[
\begin{tikzcd}
 \{1\}\arrow[Subseteq]{r}{}& \diamond & D(\langle m\rangle)\amalg (T\setminus f_Z^{-1}(\diamond))\arrow[l, "f_M"']\arrow[d, "\overline{\xi}"] \\
 \{1\}\arrow[Subseteq]{r}{}& \diamond\arrow[u, "\id"] & \arrow[l,"f_X"] A
\end{tikzcd}
\]
in $\Omega_M^E$. 

Since all of the left-hand vertical morphisms are required to be identities, we only need to check that $\overline{\xi}\circ \overline{g}=\overline{\ell}$, which is true by construction. The requirement that $\overline{\xi}$ define a morphism in $\Omega_M^E$ uniquely determines $\overline{\xi}$ on $D(\langle m\rangle)$ and the requirement that $\overline{\xi}\circ \overline{g}=\overline{\ell}$ uniquely determines $\overline{\xi}$ on $T\setminus f_Z^{-1}(\diamond)$. 
\end{proof}

\begin{cor}
The functor $\mathcal{L}$ is an $\infty$-categorical localization of $\Omega$ at the morphisms of $E$. 
\end{cor}

\begin{proof}
This follows again from \cite[Lemma 3.1.1]{Walde}. \cref{prop:LambdaWeakFibers} shows that the weak fibers $\Omega_M^E$ have initial objects, and \cref{prop:LambdaLaxFibOverAss} and  \cref{prop:LambdaLaxFibOverLambda}  show that the inclusions 
\[
\Omega_M^E\subset \Omega_{/M}
\]
are cofinal.
\end{proof}

We now rephrase the conditions from \cref{cor:CYalgconds} in terms of functors from $\Lambda^\star$. Note that by forgetting degerate intervals and localizing along $\mathcal{L}$, we have already dealt with conditions \ref{CYcond:emptysubsets} and \ref{CYcond:Eloc}. 

\begin{const}
Given $\langle n\rangle$ in $\Lambda^\star$, we define a morphism 
\[
\sigma_n:\langle n\rangle \to \{[1]_{(i,i+1)}\}_{(i,i+1)\in D(\langle n\rangle)} 
\]
 in $\Lambda^\star$ as follows. Take the canonical cyclic order on $D(\langle n\rangle)$, and define 
\[
\bigcup\nolimits^{D(\langle n\rangle)} [1]_{(i,i+1)}\to \langle n\rangle
\]
sending 
\begin{align*}
0\in [1]_{(i,i+1)} & \mapsto i\\
1\in [1]_{(i,i+1)} & \mapsto i+1.
\end{align*}
Note that given an object $X\in \Omega_{\langle n\rangle}$ in the fiber over $\langle n\rangle$, $\sigma_n$ is simply the image of the source morphism in $\Omega$.

Similarly, given an object $\{[m_i]\}_{i\in S}$ in $\Lambda^\star$, define two morphisms 
\begin{align*}
t_{\{m_i\}}:\{[m_i]\}_{i\in S} & \to \{[1]_i\}_{i\in S}\\
s_{\{m_i\}}: \{[m_i]\}_{i\in S} & \to \{[1]_{(j,j+1)}\}_{(j,j+1)\in \bigoplus_{i\in S} I([m_i])}
\end{align*}
in $\Lambda^\star$ as follows. We define $t_{\{m_i\}}:=(\id_S,\{f_i\})$ where $f_i: [1]_i\to [m_i]$ is given by the formula
\begin{align*}
f_i(0)& =0\\
f_i(1)& =m_i
\end{align*}
We define $s_i:=(\phi,\{g_{i}\})$, where
\[
\phi:\bigoplus_{i\in S}I([m_i])\to S
\]
sends $I([m_i])$ to $i$, and the morphism 
\[
g_i:\bigoplus_{(j,j+1)\in\bigoplus_{i\in S}I([m_i])} [1]_{(j,j+1)} \to [m_i] 
\]
is given by
\begin{align*}
g_i(0\in [1]_{(j,j+1)}) &= j\\
g_i(1\in [1]_{(j,j+1)}) &= j+1\\
\end{align*}
Note that, given an object $X\in \Omega_M$ in the fiber over $M:=\{[m_i]\}_{i\in S}$, the morphisms $s_M$ and $t_M$ are simply the images under $\mathcal{L}$ of the source and target morphisms, respectively.
\end{const}

\begin{lem}\label{lem:LambdaStarPullbacks}
Given a functor $G: \Lambda^\star\to \CC$, $G\circ \mathcal{L}$ satisfies condition \ref{CYcond:segal} if and only if the following two conditions on $G$ are satisfied:
\begin{enumerate}
\item For any $\{[m_i]\}_{i\in S}$, and any $\{[n_{(j,j+1)}]\}_{(j,j+1)\in\bigoplus_{i\in S} I([m_i])}$ the diagram
\[
\begin{tikzcd}
 & \{\bigstar_{(j,j+1)\in I([m])}[n_{(j,j+1)}]\}_{i\in S}\arrow[dl]\arrow[dr] & \\
\{[n_{(j,j+1)}]\}_{(j,j+1)\in\bigoplus_{i\in S}}\arrow[dr, "t_N"']& & \{[m_i]\}_{i\in S}\arrow[dl,"s_M"]  \\
 & \{[1]_{(j,j+1)}\}_{(j,j+1)\in \bigoplus_{i\in S} I([m_i])} &
\end{tikzcd}
\]
is sent to a pullback under $G$.
\item For $\langle n\rangle$ and any $\{[m_{(j,j+1)}\}_{(j, j+1)\in D(\langle n\rangle)}$ the diagram
\[
\begin{tikzcd}
 & C(\bigstar_{(j,j+1)\in D(\langle n\rangle)} [m_{(j,j+1)}])\arrow[dl]\arrow[dr] & \\
\{[m_{(j,j+1)}\}_{(j, j+1)\in D(\langle n\rangle)}\arrow[dr, "t_N"']& & \langle n\rangle\arrow[dl,"\sigma_M"]  \\
 & \langle n\rangle I([m_i]) &
\end{tikzcd}
\]
is sent to a pullback diagram under $G$.
\end{enumerate} 
\end{lem}
\begin{proof}
Using the same technique as in the proof of \cref{prop:DeltaAlgCondsinStar}, we can reduce condition \ref{CYcond:segal} to a statement about pullback squares along source and target maps. The diagrams of the lemma are then the images under $\mathcal{L}$ of  the requisite pullback diagrams. 
\end{proof}

\begin{defn}\label{defn:LambdaStarAlg}
We denote by $\Fun^{\on{alg}}(\Lambda^\star, \CC)$ the full sub-category on those functors which 
\begin{enumerate}
\item\label{LambdaStarcond:product} Send $\{[m_i\}_{i\in S}$ together with the projections to $[m_i]$ to product diagrams. 
\item\label{LambdaStarcond:Segal} Send the diagrams from \ref{lem:LambdaStarPullbacks} to pullback diagrams.
\item \label{LambdaStarcond:loc} Send the morphisms $\{[n]\}\to \langle n\rangle$ to equivalences.
\end{enumerate}
\end{defn}

\begin{cor}
There is an equivalence of $\infty$-categories 
\[
\SAlg^\CY(\CC)\simeq \Fun^{\on{Alg}}(\Lambda^\star, \CC).
\]
\end{cor}

\subsection{Extension and restriction}

\begin{defn}
We define a category $\Lambda_\Delta$ to be the Grothendieck construction of the functor 
\[
\Delta^1\overset{\{K\}}{\to} \Cat.
\]
explicitly, $\on{ob}(\Lambda_\Delta)=\on{ob}(\Lambda)\amalg \on{ob}(\Delta)$, with morphisms 
\begin{itemize}
\item $f:[n]\to [m]$ morphism in $\Delta$
\item $f:\langle n\rangle \to \langle m\rangle$ morphism in $\Lambda$
\item $f:[n]\to \langle m\rangle$ given by a morphism $f:K([n])\to \langle m\rangle$ in $\Lambda$. 
\end{itemize}

The category $(\Lambda_\Delta)^\op$ can be identified with the full subcategory of $\Lambda^\star$ on the objects $\{[m]\}$ and $\langle n\rangle$. 
\end{defn}

\begin{const}\label{const:LambdaKE}
By taking restriction and right Kan extension along the inclusion $\Lambda_\Delta^\op\subset \Lambda^\star$, we get an adjunction 
\[
\iota_\ast \Fun(\Lambda^\star, \CC) \leftrightarrow \Fun(\Lambda_\Delta^\op, \CC): \iota_!
\]
of $\infty$-categories.
\end{const}

\begin{defn}
Denote by $\Fun^\times (\Lambda^\star,\CC)$ the full $\infty$-subcategory of $\Fun (\Lambda^\star,\CC)$ on those functors which satisfy \cref{LambdaStarcond:product} from \cref{defn:LambdaStarAlg}.
\end{defn}

\begin{prop} \label{prop:LambdaKE}
The adjunction of \cref{const:LambdaKE} restricts to an equivalence of $\infty$-categories 
\[
\Fun^\times (\Lambda^\star, \CC)\simeq \Fun(\Lambda_\Delta, \CC)
\]
\end{prop}

\begin{proof}
Since there are no morphisms $\{[m_i]\}_{i\in S}\to \langle n\rangle$ in $\Lambda^\star$, this is, {\itshape mutatis mutandis}, the same as the proof of \ref{prop:DeltaKE}.
\end{proof}

\begin{const}\label{const:reflectiveloc}
We have a full subcategory $F:\Lambda\subset \Lambda_\Delta$. We can similarly define a functor 
\[
H:\Lambda_\Delta\to \Lambda
\] 
by acting as $K$ on $\Delta$ and as the identity on all other objects and morphisms. This defines an adjunction
\[
F: \Lambda \leftrightarrow\Lambda_\Delta: H
\]
It is easy to see that $H$ is a reflective localization at the morphisms $[n]\to \langle n\rangle$ given by isomorphisms $K([n])\cong \langle n\rangle$. 
\end{const}

\begin{prop}
There is an equivalence of $\infty$-categories 
\[
\Fun^{\on{alg}}(\Lambda^\star,\CC)\simeq \tSeg_\Lambda(\CC).
\]
\end{prop}

\begin{proof}
\cref{prop:LambdaKE} and \cref{const:reflectiveloc} show us that $\Fun(\Lambda^{\op},\CC)$ is equivalent, as an $\infty$-category, to the full subcategory of $\Fun^\times (\Lambda^\star,\CC)$ satisfying \ref{LambdaStarcond:product} and \ref{LambdaStarcond:loc} from \cref{defn:LambdaStarAlg}. The relation between the 2-Segal condition and condition \ref{LambdaStarcond:Segal} from \cref{defn:LambdaStarAlg} follows from a similar argument to the proof of \cref{prop:DeltaTwoSegal}.
\end{proof}

We can then summarize our results in the following theorem:

\begin{thm}\label{thm:CYAlgsSpan}
There is an equivalence of $\infty$-categories 
\[
\SAlg^\CY(\CC)\simeq \tSeg_\Lambda(\CC).
\]
\end{thm}

\newpage
\printbibliography
\end{document}